\documentclass[]{scrartcl}

\usepackage[utf8]{luainputenc}
\usepackage[USenglish]{babel}
\usepackage{csquotes}

\usepackage[a4paper,top=27mm,bottom=20mm,inner=25mm,outer=20mm]{geometry}

\usepackage[%
  backend=bibtex,bibencoding=ascii,
  style=numeric-comp,
  giveninits=true, uniquename=init, 
  natbib=true,
  url=true,
  doi=true,
  isbn=false,
  backref=false,
  maxnames=99,
  ]{biblatex}
\usepackage{amsmath}
\allowdisplaybreaks
\numberwithin{equation}{section}
\usepackage{amssymb}
\usepackage{commath}
\usepackage{mathtools}
\usepackage{bbm}
\usepackage{nicefrac}
\usepackage{subdepth}

\usepackage{siunitx}
\usepackage{algorithm}
\usepackage{algpseudocode}

\usepackage{amsthm}

\usepackage[plainpages=false,pdfpagelabels,hidelinks,unicode]{hyperref}

\usepackage{cleveref}

\usepackage{thmtools}
\usepackage{etoolbox}
\makeatletter
\patchcmd{\thmt@setheadstyle}
 {\bgroup\thmt@space}
 {\thmt@space}
 {}{}
\patchcmd{\thmt@setheadstyle}
 {\egroup\fi}
 {\fi}
 {}{}
\makeatother
\declaretheoremstyle[
  bodyfont=\normalfont\itshape,
  headformat=\NAME\ \NUMBER\NOTE,
]{myplain}
\declaretheoremstyle[
  headformat=\NAME\ \NUMBER\NOTE,
]{mydefinition}
\newcommand{\envqed}{{\lower-0.3ex\hbox{$\triangleleft$}}}
\declaretheorem[style=myplain,numberwithin=section]{theorem}

\declaretheorem[style=mydefinition,numberlike=theorem,qed=\envqed]{definition}
\declaretheorem[style=mydefinition,numberlike=theorem,qed=\envqed]{remark}

\usepackage{color}
\usepackage{graphicx}
\usepackage[small]{caption}
\usepackage{subcaption}

\begingroup\expandafter\expandafter\expandafter\endgroup
\expandafter\ifx\csname pdfsuppresswarningpagegroup\endcsname\relax
\else
\fi

\usepackage{booktabs}
\usepackage{rotating}
\usepackage{multirow}

\usepackage{enumitem}

  \usepackage[T1]{fontenc}
\usepackage{newpxtext,newpxmath}

\let\epsilon\varepsilon
\let\phi\varphi
\let\rho\varrho

\renewcommand{\i}{\mathrm{i}}

\providecommand\R{}
\renewcommand{\R}{\mathbb{R}}

\newcommand{\I}{\operatorname{I}}

\renewcommand{\vec}[1]{\pmb{#1}}

\newcommand{\tL}{\vec{t}_L}
\newcommand{\tR}{\vec{t}_R}
\newcommand{\dL}{\vec{d}_L}
\newcommand{\dR}{\vec{d}_R}

\newcommand{\mass}{\mathcal{M}}
\newcommand{\energy}{\mathcal{E}}

\newcommand{\orcid}[1]{ORCID:~\href{https://orcid.org/#1}{#1}}
\usepackage{authblk}

\newenvironment{keywords}{\par\textbf{Key words.}}{\par}
\newenvironment{AMS}{\par\textbf{AMS subject classification.}}{\par}

\title{High-order mass- and energy-conserving methods for the nonlinear Schr{\"o}dinger equation and its hyperbolization}

\author[1]{Hendrik~Ranocha\thanks{\orcid{0000-0002-3456-2277}}}
\affil[1]{Institute of Mathematics, Johannes Gutenberg University Mainz, Staudingerweg 9, 55128 Mainz, Germany}

\author[2]{David~I.~Ketcheson\thanks{\orcid{0000-0002-1212-126X}}}
\affil[2]{King Abdullah University of Science and Technology (KAUST),
Computer Electrical and Mathematical Science and Engineering Division (CEMSE),
Thuwal, 23955-6900, Saudi Arabia}

\date{June 15, 2026} 

\makeatletter
\hypersetup{pdfauthor={Hendrik Ranocha, David I. Ketcheson}} 
\hypersetup{pdftitle={High-order mass- and energy-conserving methods for the nonlinear Schr{\"o}dinger equation and its hyperbolization}} 
\makeatother

\begin{document}

\maketitle

\begin{abstract}
We propose a class of numerical methods for the nonlinear Schrödinger (NLS)
equation that conserves mass and energy, is of arbitrarily high-order accuracy
in space and time, and requires only the solution of a nonlinear scalar algebraic
equation per time step.  We show that some existing spatial discretizations,
including the popular Fourier spectral method, are in fact
energy-conserving if one considers the appropriate form of the energy density.
We develop a new relaxation-type approach
for conserving multiple nonlinear functionals that is more efficient and robust
for the NLS equation compared to the existing multiple-relaxation approach.
The accuracy and efficiency of the new schemes is demonstrated on test problems
for both the focusing and defocusing NLS.

\end{abstract}



\section{Introduction}

We study discretizations of the nonlinear Schr{\"o}dinger (NLS) equation (also referred to as Gross-Pitaevskii equation)
\begin{equation}
\label{eq:nls-u}
  \i u_t + u_{xx} + \beta |u|^2 u = 0,\qquad t > 0, x \in (x_L, x_R),
\end{equation}
with periodic or homogeneous Neumann boundary conditions and given
initial data $u(0, x) = u^0(x)$, where $\beta$ is a real parameter.
The case $\beta > 0$ is called focusing, while $\beta < 0$ is defocusing.
The NLS equation arises in important physical applications, including
deep water waves, Bose-Einstein condensates, and others.  It is one of the most
widely-studied examples of an integrable nonlinear PDE, and possesses a countably-infinite
set of conserved functionals. Two of the most significant are the mass $\mass$ and energy $\energy$:
\begin{equation} \label{conserved-quantities}
  \mass := \int |u|^2 \dif x, \qquad
  \energy := \int \left(|u_x|^2 - \frac{\beta}{2}|u|^4\right) \dif x.
\end{equation}
For brevity, we will call a numerical scheme doubly-conservative if it conserves
discrete analogs of both of the above quantities.
Numerical conservation is extremely important
in obtaining accurate solutions of \eqref{eq:nls-u}.
This is especially true when dealing with higher-degree nonlinearities or in higher dimensions,
where it is possible for solutions to exhibit finite-time blowup, which depends
critically on the value of the energy.  In such cases, numerical conservation of
the mass and energy is essential in order to obtain even qualitatively
correct solutions.  However, even conservative methods can struggle to
accurately capture some solutions, such as multi-soliton bound states
that exhibit large derivatives \cite{herbst1985numerical}.
Thus, the ideal discretization would conserve the mass and energy and be
high-order accurate while also being computationally efficient.

The earliest proposed doubly-conservative (mass- and energy-conserving) schemes were
second-order accurate and fully implicit,
requiring the solution of large systems of nonlinear algebraic
equations \cite{delfour1981finite,sanz1984methods}. These methods have been
analyzed later when combined with finite element schemes in space,
e.g., \cite{akrivis1991fully,henning2017crank}.
Sanz-Serna and Manoranjan proposed an explicit scheme that used a modified time
step size to ensure fully-discrete conservation of the mass only
\cite{sanz1983method}; that scheme can be seen as a precursor to the time relaxation
approaches that have been developed more recently.
Besse \cite{besse2004relaxation} proposed the first doubly-conservative scheme that
does not require any nonlinear algebraic solve, as it is only linearly
implicit; however, it is limited to second-order accuracy in time.
The scalar auxiliary variable (SAV) approach has been used to conserve both the mass and a modified energy with fully implicit methods \cite{cui2021mass}.
There are also linearly-implicit second-order SAV methods that conserve the mass and a modified energy \cite{antoine2021scalar,deng2021second}.
Bai et al.\ \cite{bai2024high} provided a family of conservative
discretizations of arbitrarily high order in space and time, which require the
use of fully-implicit time stepping.
Akrivis et al.\ \cite{akrivis2025high} propose a remarkable fully implicit class of schemes
that conserve not only mass and energy, but also momentum.
Biswas et al.\ \cite{biswas2024accurate} proposed an approach that couples
existing mass- and momentum-conservative second-order spatial schemes with
time relaxation in order to achieve fully-discrete conservation without the
need to solve any large (linear or nonlinear) algebraic systems.
If conservation is not required, then standard high-order time discretizations
can be combined with spectral methods in space to obtain
highly accurate schemes \cite{antoine2016high}.

Among the existing approaches, we see that no scheme combines all
of the following properties:
\begin{itemize}
    \item Mass and energy conservation;
    \item Arbitrarily high order in space and time;
    \item No need to solve large algebraic systems of equations.
\end{itemize}
In the present work we propose such a class of methods for the NLS equation \eqref{eq:nls-u}
in the presence of periodic or
homogeneous Neumann boundary conditions.  Furthermore, we
provide new insight into some existing schemes that fit within
the framework of summation-by-parts (SBP) spatial discretizations.  In fact, we show
that a broad class of schemes including the popular Fourier spectral discretization,
are energy-preserving if one considers the appropriate form for the energy
(for the Fourier spectral method, this has been observed, e.g., in \cite{zhang2017conservative}).

One could alternatively achieve the three properties above by combining
the multiple-relaxation approach \cite{biswas2024accurate} with,
e.g., Fourier spectral discretization in space.  While we do not
pursue this approach here, a comparison between the method we propose
and multiple relaxation has recently been conducted for the very similar case of
the Schr{\"o}dinger-Poisson system \cite{rajvanshi2026efficient}, where the present
approach is found to be advantageous in terms of efficiency and robustness.

We formulate the mass- and energy-conserving spatial semidiscretizations
using SBP operators in Section~\ref{sec:spatial_semidiscretizations}.
References and introductions to SBP operators
are given in \cite{svard2014review,fernandez2014review}.
SBP operators can be used to formulate finite differences (FDs)
\cite{kreiss1974finite,strand1994summation,carpenter1994time},
finite volumes \cite{nordstrom2001finite},
continuous Galerkin (CG) finite elements \cite{hicken2016multidimensional,hicken2020entropy,abgrall2020analysisI},
discontinuous Galerkin~(DG) methods \cite{gassner2013skew,carpenter2014entropy},
flux reconstruction~(FR) \cite{huynh2007flux,vincent2011newclass,ranocha2016summation},
active flux methods \cite{eymann2011active,barsukow2025stability},
as well as meshless schemes \cite{hicken2024constructing}.

To construct high-order accurate time integrators that conserve mass and energy
in Section~\ref{sec:time_discretizations},
we extend the recently developed framework of relaxation-in-time methods
\cite{ketcheson2019relaxation,ranocha2020relaxation,ranocha2020general}.
The basis for these methods goes back to \cite{sanz1982explicit} and \cite[pp.~265--266]{dekker1984stability}.
It has been used for various time integration schemes, e.g.,
Runge-Kutta methods \cite{ranocha2020relaxation},
linear multistep methods \cite{ranocha2020general},
residual distribution schemes \cite{abgrall2022relaxation},
IMEX methods \cite{kang2022entropy,li2022implicit},
and multi-derivative methods \cite{ranocha2023functional,ranocha2024multiderivative}.
Applications include among others Hamiltonian problems \cite{ranocha2020relaxationHamiltonian,zhang2020highly,li2023relaxation},
kinetic equations \cite{leibner2021new},
compressible flows \cite{yan2020entropy,ranocha2020fully,doehring2025paired},
and dispersive wave equations \cite{li2025time,ranocha2025structure,lampert2024structure,mitsotakis2021conservative}.

In Section~\ref{sec:performance_comparison}, we compare the performance of
our proposed approach to other methods from the literature,
demonstrating that it is more efficient in all cases tested.
In Section~\ref{sec:hyperbolization}, we develop similar semidiscretizations
for the hyperbolized NLS equation \cite{biswas2025hyperbolic} and combine them
with the new time integrators.
Finally, we summarize and discuss our findings in Section~\ref{sec:summary}.

We have implemented all methods in Julia \cite{bezanson2017julia} using
SummationByPartsOperators.jl \cite{ranocha2021sbp} for the spatial
semidiscretizations, wrapping FFTW.jl \cite{frigo2005design} for the
Fourier collocation methods. To solve sparse linear systems, we use
UMFPACK from SuiteSparse \cite{davis2004umfpack,amestoy2004amd,davis2004colamd}
wrapped in the Julia standard library.
To solve the nonlinear relaxation equation, we use the method of
\cite{klement2014using} implemented in SimpleNonlinearSolve.jl
\cite{pal2024nonlinearsolve}.
We use CairoMakie.jl \cite{danisch2021makie} to visualize the results.
All code and data required to reproduce the numerical results is
available online \cite{ranocha2025highRepro}.

\section{Spatial semidiscretizations}
\label{sec:spatial_semidiscretizations}

We start by recalling two well-known spatial discretizations with
conservation properties.  As we will see, both of these fit into a broader framework
of SBP discretizations, which we then introduce and discuss more generally.
The main theoretical contribution of this section is the arbitrarily
high-order accurate mass- and energy-conserving semidiscretizations
introduced in Section~\ref{sec:semidiscretization}.

Separating the real and imaginary parts of $u = v + \i w$, we can rewrite
\eqref{eq:nls-u} as
\begin{equation}
\label{eq:nls-vw}
\begin{aligned}
  v_t + w_{xx} + \beta \bigl( v^2 + w^2 \bigr) w &= 0,
  \\
  w_t - v_{xx} - \beta \bigl( v^2 + w^2 \bigr) v &= 0.
\end{aligned}
\end{equation}

Given a spatial interval $[x_L, x_R]$, we consider (possibly non-uniform)
grids $\vec{x} = (x_1, \dots, x_N)^T$ with
$x_L = x_1 \le \dots$ and $x_N = x_R$ for non-periodic
boundary conditions; for periodic boundaries, $x_N$ may not coincide
with $x_R$.
We specifically allow $x_i = x_{i+1}$ for some $i$ to
formulate discontinuous Galerkin methods naturally as SBP methods.
Furthermore, we consider collocation methods, i.e., the numerical solution
is represented by point values $\vec{u} = (u_1, \dots, u_N)^T$
with $u_i \approx u(x_i)$.

\subsection{Two examples of conservative semidiscretizations} \label{sec:2semi}

For the NLS equation \eqref{eq:nls-vw} with periodic boundaries, perhaps the most widely used
spatial discretization is based on Fourier collocation
spectral differentiation, wherein the second-derivative operator
$\partial_{xx}$ is approximated on a uniform grid in space by
a symmetric operator related to the discrete Fourier transform
on a uniform grid.  Semidiscretizations based on this are known to preserve the
discrete mass (see, e.g., \cite{bao2002time}) which is given by
\begin{align} \label{eq:mass-nls}
    \mass = \vec{v}^T M \vec{v} + \vec{w}^T M \vec{w}
    \approx \int \bigl( v^2 + w^2 \bigr) \dif x
    = \int |u|^2 \dif x
\end{align}
where $\vec{v}, \vec{w}$ are the real and imaginary parts of $\vec{u}$,
respectively, and $M = \Delta x \I$ is the mass matrix obtained by multiplying
the identity matrix $\I$ by the grid spacing $\Delta x = x_{i + 1} - x_{i}$.

Next we consider a second-order-accurate finite element discretization that conserves both
discrete mass \eqref{eq:mass-nls} and the energy functional
\begin{equation}
\label{eq:energy-nls-FEM}
\begin{aligned}
  \energy
  &=
  \Delta x \sum_{i=1}^{N-1} \left( \frac{(v_{i+1} - v_i)^2}{\Delta x^2} + \frac{(w_{i+1} - w_i)^2}{\Delta x^2} \right)
  - \frac{\beta}{2} \vec{1}^T M \bigl( \vec{v}^2 + \vec{w}^2 \bigr)^2
  \\
  &\approx
  \int \left( (v_x)^2 + (w_x)^2 - \frac{\beta}{2} \bigl( v^2 + w^2 \bigr)^2 \right)
  =
  \int \left( |u_x|^2 - \frac{\beta}{2} |u|^4 \right).
\end{aligned}
\end{equation}
The scheme is \cite{herbst1985numerical}
\begin{align}\label{FEMdisc}
    \vec{v}'(t) & =   M^{-1} A_2 \vec{w} - \beta \bigl( \vec{v}^2 + \vec{w}^2 \bigr) \vec{w}, \\
    \vec{w}'(t) & = - M^{-1} A_2 \vec{v} + \beta \bigl( \vec{v}^2 + \vec{w}^2 \bigr) \vec{v},
\end{align}
where for homogeneous Neumann boundaries, the matrices $M$ and $A_2$ are given by
\begin{equation} \label{FEM-matrices}
    M = \Delta x \begin{pmatrix}
    \frac{1}{2} &  &  &  &  \\
     & 1 &  &  &  \\
     &  & {\ddots} &  &  \\
     &  &  & 1 &  \\
     &  &  &  & \frac{1}{2}
    \end{pmatrix}
    \ \ \text{,} \ \
    A_2 = \frac{1}{\Delta x}
    \begin{pmatrix}
     1 & -1 &  &  &  \\
     -1 & 2 & -1 &  & \\
     & \ddots & \ddots & \ddots & \\
     &  & -1 & 2 & -1\\
     &   &  & -1 & 1\\
    \end{pmatrix}\;.
\end{equation}
Both of the above semidiscretizations approximate the second derivative using a
symmetric matrix, and each of these matrices, combined with the corresponding mass
matrix (given by the scaled identity $M = \Delta x \I$ for the spectral method and
by $M$ in \eqref{FEM-matrices} for the finite element method), satisfy what is known as a summation-by-parts (SBP) property.

\subsection{Summation-by-parts operators}
\label{sec:sbp_operators}

In this section we review standard SBP operators and notation;
for more details see e.g.
\cite{mattsson2004summation,mattsson2017diagonal,svard2014review,fernandez2014review,ranocha2021broad}.
In non-periodic domains, the restriction of a numerical solution $\vec{u}$
to the boundaries is given by $\tL^T \vec{u}$ and $\tR^T \vec{u}$,
where
\begin{equation}
  \tL = (1, 0, \dots, 0)^T,
  \quad
  \tR = (0, \dots, 0, 1)^T.
\end{equation}
For periodic domains, we use the convention $\tL = \tR = \vec{0}$
to simplify the notation.

\begin{definition}
  A first-derivative SBP operator on the grid $\vec{x}$ is given by a
  consistent first-derivative operator $D_1 \in \R^{N \times N}$ and a
  symmetric and positive definite mass (or \emph{norm}) matrix $M \in \R^{N \times N}$
  satisfying the SBP property
  \begin{equation}
    M D_1 + D_1^T M
    =
    \tR \tR^T - \tL \tL^T,
  \end{equation}
  where $\tL, \tR \in \R^N$ are the boundary restriction operators
  described above.
\end{definition}

\begin{definition}
  First-derivative upwind SBP operators on the grid $\vec{x}$ are given by
  consistent first-derivative operators $D_\pm \in \R^{N \times N}$ and a
  symmetric and positive definite mass (or \emph{norm}) matrix $M \in \R^{N \times N}$
  satisfying the upwind SBP property
  \begin{equation}
    M D_+ + D_-^T M
    =
    \tR \tR^T - \tL \tL^T,
    \quad
    M (D_+ - D_-) \text{ is symmetric and negative semidefinite},
  \end{equation}
  where $\tL, \tR \in \R^N$ are the boundary restriction operators
  described above.
\end{definition}

\begin{definition}\label{dfn:SBP2}
  A second-derivative SBP operator on the grid $\vec{x}$ is given by a
  consistent second-derivative operator $D_2 \in \R^{N \times N}$ and a
  symmetric and positive definite mass (or \emph{norm}) matrix $M \in \R^{N \times N}$
  satisfying the SBP property
  \begin{equation}
    M D_2
    =
    \tR \dR^T - \tL \dL^T - A_2,
    \qquad
    A_2 \text{ is symmetric and negative semidefinite},
  \end{equation}
  where $\tL, \tR \in \R^N$ are the boundary restriction operators
  described above and $\dL, \dR \in \R^N$ yield consistent first-derivative
  approximations at the boundaries.
\end{definition}
It can be verified by direct computation that both of the semidiscretizations
described in Section \ref{sec:2semi} satisfy the conditions of Definition
\ref{dfn:SBP2}.

In the following, we will only use diagonal mass (or \emph{norm}) matrices $M$.
Concretely, we will use FDs, CG spectral element methods,
DG spectral element methods,
and Fourier collocation schemes. See, e.g., \cite{ranocha2021broad}
for details and additional references.

\subsection{A general doubly-conservative SBP semidiscretization}
\label{sec:semidiscretization}

We propose the general SBP semidiscretization
\begin{equation}
\label{eq:semidiscretization-nls}
\begin{aligned}
  \partial_t \vec{v}
  &=
  - \bigl( D_2 - M^{-1} \tR \dR^T + M^{-1} \tL \dL^T \bigr) \vec{w}
  - \beta \bigl( \vec{v}^2 + \vec{w}^2 \bigr) \vec{w},
  \\
  \partial_t \vec{w}
  &=
  \bigl( D_2 - M^{-1} \tR \dR^T + M^{-1} \tL \dL^T \bigr) \vec{v}
  + \beta \bigl( \vec{v}^2 + \vec{w}^2 \bigr) \vec{v},
\end{aligned}
\end{equation}
of the nonlinear Schr{\"o}dinger equation \eqref{eq:nls-vw}, where
$D_2$ is a second-derivative SBP operator with diagonal mass matrix $M$.
Recall that we use the convention $\tL = \tR = \dL = \dR = \vec{0}$
in periodic domains and abbreviate
\begin{equation}
  D_2 - M^{-1} \tR \dR^T + M^{-1} \tL \dL^T
  =
  -M^{-1} A_2.
\end{equation}
For CG spectral element methods with homogeneous Neumann
boundary conditions on a uniform mesh, this SBP semidiscretization is exactly
\eqref{FEMdisc}. For Fourier collocation methods, it is the standard
semidiscretization using the Fourier spectral second-derivative operator
described in Section~\ref{sec:2semi}.

In order to discuss the energy conservation property for \eqref{eq:semidiscretization-nls},
we introduce a different form of the discrete energy:
\begin{equation}
\label{eq:energy-nls}
\begin{aligned}
  \energy
  &=
  \vec{v}^T A_2 \vec{v}
  + \vec{w}^T A_2 \vec{w}
  - \frac{\beta}{2} \vec{1}^T M \bigl( \vec{v}^2 + \vec{w}^2 \bigr)^2
  \\
  &\approx
  \int \left( (v_x)^2 + (w_x)^2 - \frac{\beta}{2} \bigl( v^2 + w^2 \bigr)^2 \right)
  =
  \int \left( |u_x|^2 - \frac{\beta}{2} |u|^4 \right).
\end{aligned}
\end{equation}
In fact, \eqref{eq:energy-nls-FEM} is of the form \eqref{eq:energy-nls} since
\begin{equation}
\begin{aligned}
  &\quad
  \sum_{i=1}^{N-1} (v_{i+1} - v_i)^2
  =
  \sum_{i=1}^{N-1} (v_{i+1}^2 - 2 v_i v_{i+1} + v_i^2)
  =
  v_1^2 + 2 \sum_{i=2}^{N-1} v_i^2 + v_N^2
  - 2 \sum_{i=1}^{N-1} v_i v_{i+1}
  \\
  &=
  v_1^2 + 2 \sum_{i=2}^{N-1} v_i^2 + v_N^2
  - \sum_{i=1}^{N-1} v_i v_{i+1} - \sum_{i=2}^N v_i v_{i-1} \\
  & =
  \sum_{i=2}^{N-1} v_i (-v_{i+1} + 2 v_i - v_{i-1})
  + v_1 (v_1 - v_2)
  + v_N (v_N - v_{N-1})
  =
  \Delta x \, \vec{v}^T A_2 \vec{v}.
\end{aligned}
\end{equation}
The form \eqref{eq:energy-nls} using the discrete second-derivative operator
can be generalized more easily to general SBP operators. Thus, we will use
\eqref{eq:energy-nls} in the following.
We will see in Section~\ref{sec:kinetic_energy_comparison} that this
choice of the discrete total energy is crucial.

\begin{theorem}
\label{thm:semidiscrete_conservation_nls}
  The semidiscretization \eqref{eq:semidiscretization-nls} conserves the
  discrete total mass $\mass$ \eqref{eq:mass-nls} and the discrete
  total energy $\energy$ \eqref{eq:energy-nls} for diagonal-norm SBP
  operators.
\end{theorem}
\begin{proof}
  Since the mass matrix $M$ is symmetric, we have
  \begin{equation}
  \begin{aligned}
    \partial_t \mass
    &=
    2 \vec{v}^T M \partial_t \vec{v}
    + 2 \vec{w}^T M \partial_t \vec{w}
    \\
    &=
    - 2 \vec{v}^T A_2 \vec{w}
    - 2 \beta \vec{v}^T M \bigl( \vec{v}^2 + \vec{w}^2 \bigr) \vec{w}
    + 2 \vec{w}^T A_2 \vec{v}
    + 2 \beta \vec{w}^T M \bigl( \vec{v}^2 + \vec{w}^2 \bigr) \vec{v}
    = 0,
  \end{aligned}
  \end{equation}
  where we have used the symmetry of $A_2$ and the fact that
  the mass matrix $M$ is diagonal. For the energy, we compute
  \begin{equation}
  \begin{aligned}
    \partial_t \energy
    &=
    2 \vec{v}^T A_2 \partial_t \vec{v}
    + 2 \vec{w}^T A_2 \partial_t \vec{w}
    - 2 \beta \vec{1}^T M \bigl( \vec{v}^2 + \vec{w}^2 \bigr) \vec{v} \partial_t \vec{v}
    - 2 \beta \vec{1}^T M \bigl( \vec{v}^2 + \vec{w}^2 \bigr) \vec{w} \partial_t \vec{w}
    \\
    &=
    2 \vec{v}^T A_2 M^{-1} A_2 \vec{w}
    - 2 \beta \vec{v}^T A_2 \bigl( \vec{v}^2 + \vec{w}^2 \bigr) \vec{w}
    - 2 \vec{w}^T A_2 M^{-1} A_2 \vec{v}
    \\
    &\quad
    + 2 \beta \vec{w}^T A_2 \bigl( \vec{v}^2 + \vec{w}^2 \bigr) \vec{v}
    - 2 \beta \vec{1}^T \bigl( \vec{v}^2 + \vec{w}^2 \bigr) \vec{v} A_2 \vec{w}
    + 2 \beta^2 \vec{1}^T M \bigl( \vec{v}^2 + \vec{w}^2 \bigr)^2 \vec{v} \vec{w}
    \\
    &\quad
    + 2 \beta \vec{1}^T \bigl( \vec{v}^2 + \vec{w}^2 \bigr) \vec{w} A_2 \vec{v}
    - 2 \beta^2 \vec{1}^T M \bigl( \vec{v}^2 + \vec{w}^2 \bigr)^2 \vec{v} \vec{w}
    \\
    &=
    0,
  \end{aligned}
  \end{equation}
  where we have again used the symmetry of $A_2$ and the fact that
  the mass matrix $M$ is diagonal.
\end{proof}

\begin{remark}
\label{rem:diagonal_norm_necessary}
  The requirement of a diagonal mass (or \emph{norm}) matrix $M$ is necessary in the proof of Theorem~\ref{thm:semidiscrete_conservation_nls}.
  Without this assumption, the last steps showing conservation of the total mass and energy cannot be performed in general since, e.g.,
  \begin{equation*}
    \vec{v}^T M \bigl( \vec{v}^2 + \vec{w}^2 \bigr) \vec{w}
    \neq
    \vec{w}^T M \bigl( \vec{v}^2 + \vec{w}^2 \bigr) \vec{v}.
  \end{equation*}
\end{remark}

\subsection{Importance of the kinetic energy discretization}
\label{sec:kinetic_energy_comparison}

It is crucial to choose the discretization of the kinetic energy using the second-derivative SBP operator as in \eqref{eq:energy-nls}.
To demonstrate this, we consider the two-soliton setup described in
\cite{biswas2024accurate}
and discretize it with a Fourier collocation method using $N$ nodes.
We compute the absolute changes of the total mass \eqref{eq:mass-nls}, the total energy \eqref{eq:energy-nls}
(computed correctly using the second-derivative SBP operator),
and a naive version of the total energy where the kinetic energy terms are computed as
\begin{equation} \label{eq:naive}
  \vec{v}^T D_1^T M D_1 \vec{v} \approx \int |v_x|^2,
\end{equation}
where $M = \Delta x \I$ is the usual Fourier collocation mass matrix and
$D_1$ the corresponding first-derivative operator.

The second-derivative FFT-based operator $D_2$ is the
same as the square of the first-derivative FFT-based operator $D_1$
if and only if the number of nodes $N$ is odd \cite{johnson2011notes}.
Thus, we expect to see some clear differences for even $N$ (which
is usually chosen in practice for efficiency) but not for odd $N$
(which is rarely used in practice, if at all).

\begin{figure}[htb]
  \centering
  \includegraphics[width=\textwidth]{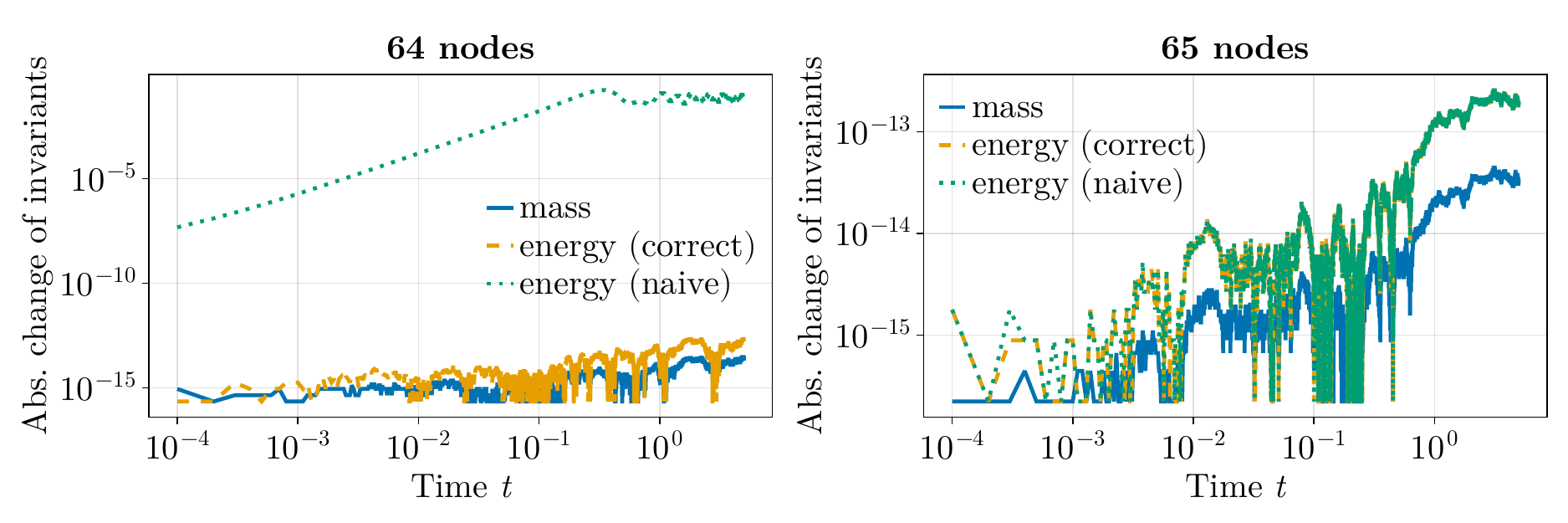}
  \caption{Absolute changes of the discretized invariants of the
           nonlinear Schr{\"o}dinger equation for Fourier collocation
           methods with $N \in \{64, 65\}$ nodes in the domain $[-35, 35]$
           for the two-soliton setup. The semidiscretization is
           integrated in time with the fifth-order method of
           \cite{kennedy2019higher} and time step size
           $\Delta t = 10^{-4}$.}
  \label{fig:kinetic_energy_comparison}
\end{figure}

The results shown in Figure~\ref{fig:kinetic_energy_comparison}
support the theoretical prediction.   Here we use very small time
steps in order to ensure that the impact of time integration on the
conservation properties is insignificant.
For odd $N$, we observe near
machine-precision conservation of both approximations of the energy.
In contrast, for even $N$, we see that the
naive discretization \eqref{eq:naive} of the total energy varies significantly,
while the approximation \eqref{eq:energy-nls} is still conserved.

Indeed, the total mass and the
correctly discretized total energy are conserved up to the error
of the time integrator, which is close to machine precision due
to sufficiently small time step sizes to highlight the spatial effects.

\subsection{Numerical verification of the semidiscrete invariant conservation}
\label{sec:semidiscrete_conservation}

Having verified the mass and energy conservation for Fourier collocation
methods, we next demonstrate these structure-preserving properties for
FD, CG, and local DG (LDG)
SBP operators in periodic and bounded domains.
As in Section~\ref{sec:kinetic_energy_comparison}, we choose the
two-soliton setup, coarse spatial resolutions, and highly resolved
time discretizations.
The results shown in Figure~\ref{fig:semidiscrete_conservation} support
the analysis, i.e., the total mass and the total energy are conserved
up to a small error caused by the time integrator.

\begin{figure}[htb]
  \centering
  \includegraphics[width=\textwidth]{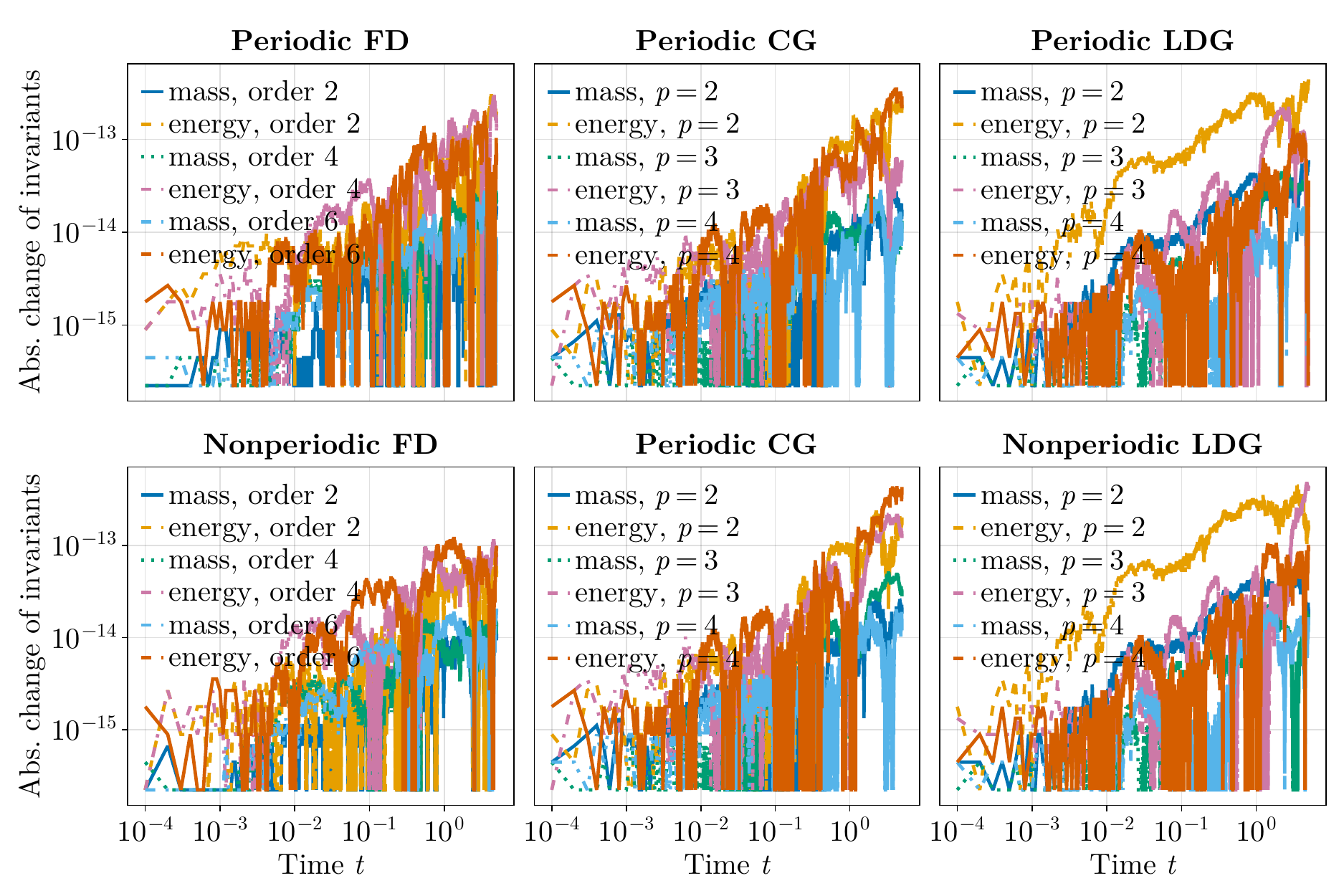}
  \caption{Absolute changes of the discretized invariants of the
           nonlinear Schr{\"o}dinger equation for the two-soliton setup.
           The semidiscretizations are integrated in time with the
           fifth-order method of \cite{kennedy2019higher} and time step
           size $\Delta t = 10^{-4}$. For all semidiscretizations, we
           use approximately 64 degrees of freedom with slight deviations
           depending on the polynomial degree $p$ for Galerkin methods.}
  \label{fig:semidiscrete_conservation}
\end{figure}

\subsection{Convergence tests in space}

Next, we use the two-soliton setup again to verify the high-order accuracy of the spatial semidiscretizations.
We choose a sufficiently small time step size $\Delta t = 5 \times 10^{-5}$ to make sure that the error is dominated by the spatial discretization.

\begin{figure}[htb]
  \centering
  \includegraphics[width=\textwidth]{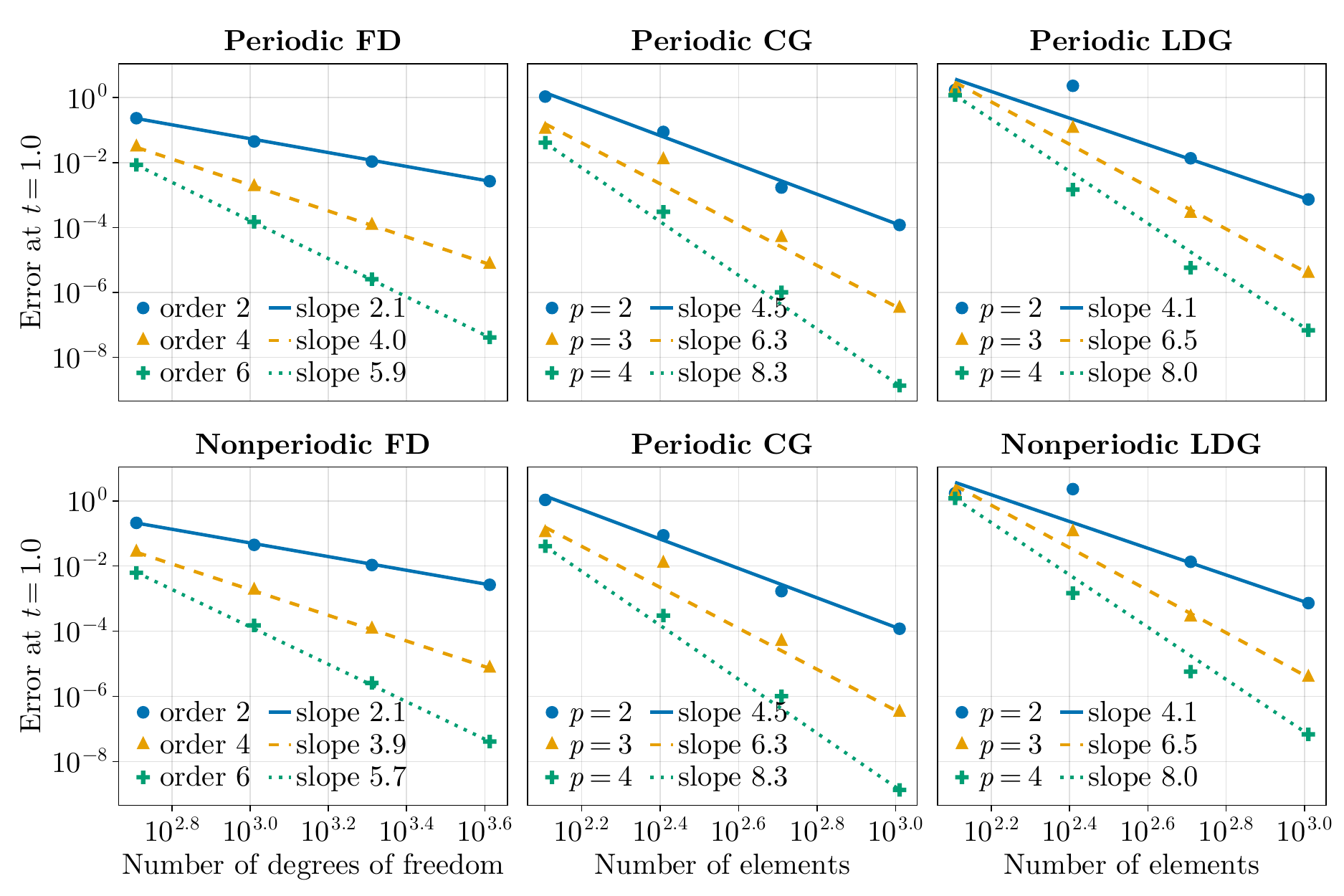}
  \caption{Spatial convergence study for the two-soliton setup.
           The semidiscretizations are integrated in time with the
           fifth-order method of \cite{kennedy2019higher} and time step
           size $\Delta t = 5 \times 10^{-5}$.
           For finite difference methods, the ``order'' refers to the
           order of accuracy in the interior. For Galerkin methods,
           $p$ is the polynomial degree.}
  \label{fig:convergence_in_space}
\end{figure}

The results shown in Figure~\ref{fig:convergence_in_space} confirm the high accuracy of the spatial semidiscretizations.
For all methods, we measured the error in the discrete $L^2$ norm induced by the (diagonal) mass matrix $M$.
For periodic finite differences, we observe the expected order of convergence determined by the order of accuracy.
For finite differences with homogeneous Neumann boundary conditions, we also observe the same order of convergence, although SBP operators typically lead to a reduced order of convergence due to the lower order of accuracy at the boundaries \cite{svard2019convergence,svard2020convergence}.
However, we have chosen the domains large enough so that the interaction with the boundaries is negligible (since the soliton solutions are only valid on the whole real line).

Both CG and DG methods show a superconvergence rate between $2p$ and $2p + 1/2$, where $p$ is the polynomial degree.
The reason for this superconvergence is unknown and beyond the scope of the present work.
Similar superconvergence effects have been observed and studied for Galerkin-type methods in several cases, e.g., \cite{douglas1973superconvergence,cao2017optimal}.

\subsection{Dispersive shock wave}

Next, we consider a dispersive shock wave for the defocusing NLS with $\beta =
-1$ following \cite{gurevich1987dissipationless,dhaouadi2019extended}.  We use
a Riemann problem as the initial condition:
\begin{equation}
  \varrho = \begin{cases}
    \varrho_L, & \text{if } x < 0, \\
    \varrho_R, & \text{otherwise},
  \end{cases}
  \qquad
  \theta = 0,
\end{equation}
where the hydrodynamic variables $\varrho$ and $\theta$ are related to $u$ via the Madelung transformation
\begin{equation}
  u = \sqrt{\varrho} \exp(\i \theta).
\end{equation}
As in \cite[Section~5.2]{dhaouadi2019extended}, we use the parameters $\varrho_L = 2$ and $\varrho_R = 1$.

To avoid interactions with the boundaries while focusing on the dispersive shock wave, we consider the large domain $[-1600, 1600]$ and show the numerical solution at time $t = 100$ for $-400 < x < 500$.
We use Fourier collocation methods with $N = 2^{16}$ nodes in space and the fifth-order method of \cite{kennedy2019higher} with time step size $\Delta t = 0.05$.
To plot the hydrodynamic variables $\varrho = |u|^2$ and $v = \theta_x$, we compute the argument of $u$ using the two-argument arc tangent, correct branch cuts, and differentiate the result numerically.

\begin{figure}[htb]
  \centering
  \includegraphics[width=\textwidth]{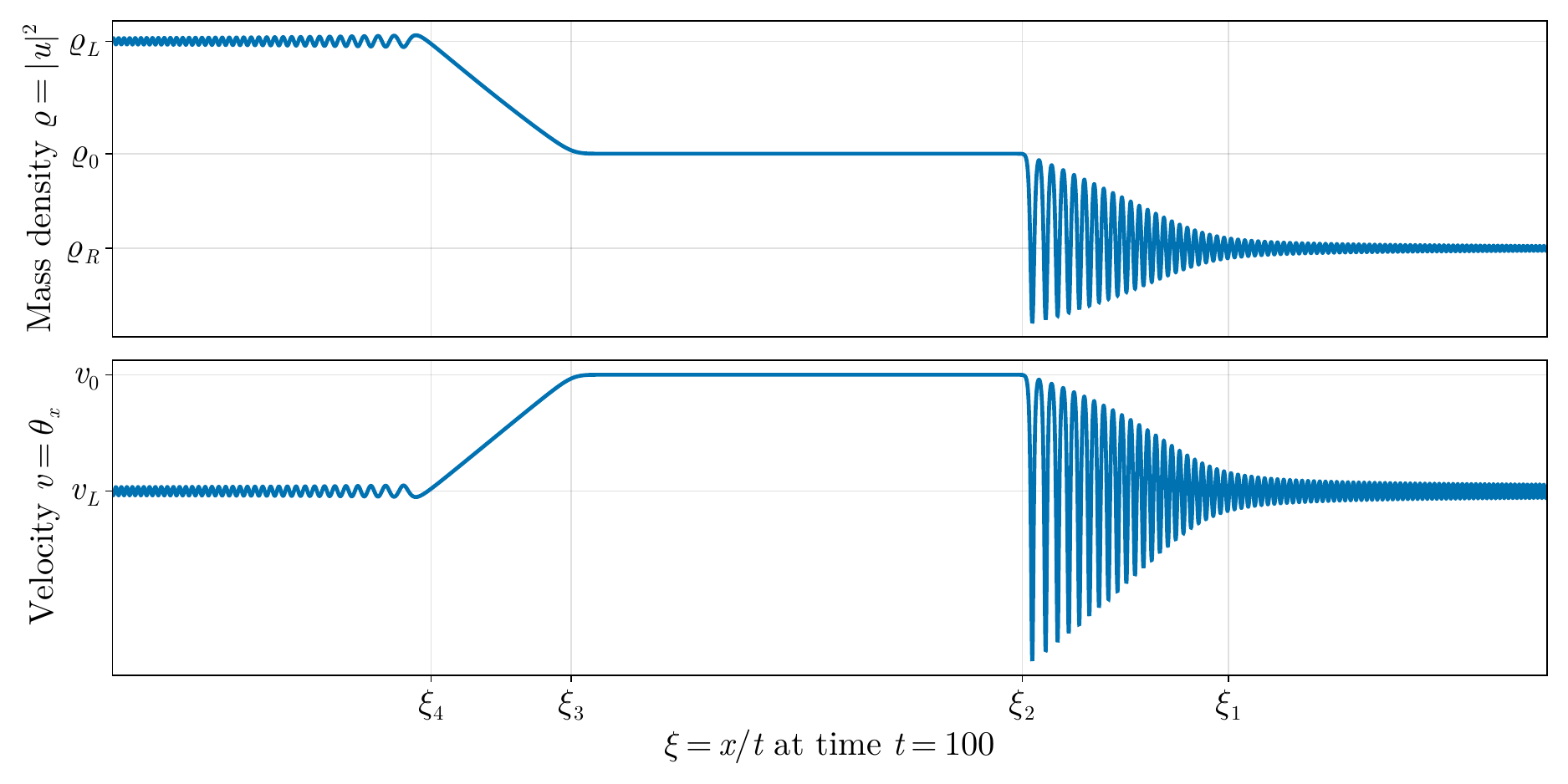}
  \caption{Numerical results for the dispersive shock wave obtained
           with Fourier collocation methods and the fifth-order method
           of \cite{kennedy2019higher}.}
  \label{fig:dispersive_shock_wave}
\end{figure}

The results shown in Figure~\ref{fig:dispersive_shock_wave} are in good agreement with the numerical and analytical results of \cite{dhaouadi2019extended}.
In particular, the asymptotic values $\varrho_L$, $\varrho_R$, and $v_L = 0 = v_R$ as well as the intermediate state $(\varrho_0, v_0)$ are captured correctly.
Furthermore, the extent of the waves in the self-similarity variable $\xi = x / t$ agrees well with the theoretical predictions based on Whitham's modulation equations \cite[Section~2.2.2]{dhaouadi2019extended}.

\subsection{Rarefaction problem}
\label{sec:rarefaction_problem}

We consider a rarefaction problem for the focusing NLS with $\beta = 1$ following \cite{wang2025asymptotic}; see also \cite{dong2026numerical}.
The Riemann initial data are given by $u = \exp(\i \mu |x|)$ with $\mu = 1.5 \sqrt{2}$.
We use the Fourier collocation methods with $N = 2^{14}$ nodes in the computational domain $[-500, 500]$ and show the numerical solution at time $t = 15$ for $-120 < x < 120$.
Time integration is performed with the fifth-order method of \cite{kennedy2019higher} and time step size $\Delta t = 0.05$.

\begin{figure}[htb]
  \centering
  \includegraphics[width=\textwidth]{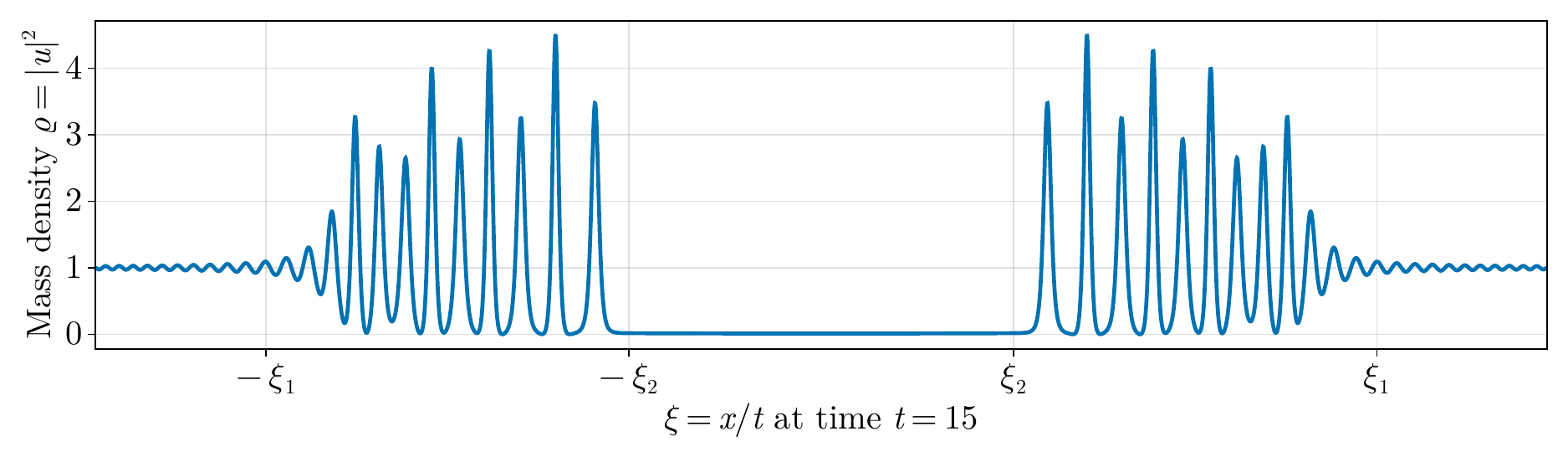}
  \caption{Numerical results for the dispersive rarefaction wave obtained with Fourier collocation methods and the fifth-order method of \cite{kennedy2019higher}.
           The theoretical results of \cite{wang2025asymptotic} predict a plane wave region for $|\xi| > \xi_1$, an oscillatory wave region for $\xi_2 < |\xi| < \xi_1$, and a vacuum region for $|\xi| < \xi_2$.}
  \label{fig:dispersive_rarefaction}
\end{figure}

The results shown in Figure~\ref{fig:dispersive_rarefaction} agree well with the numerical results of \cite{wang2025asymptotic,dong2026numerical} and the analytical results based on Whitham modulation theory and Riemann-Hilbert problems of \cite{wang2025asymptotic}.
Please note that we show the mass density $|u|^2$ instead of $|u|$ and use the same self-similarity variable $\xi = x / t$ instead of $t$ in accordance with earlier plots in this article, in contrast to the plots in \cite{wang2025asymptotic,dong2026numerical}.

\section{Time discretizations}
\label{sec:time_discretizations}

Many differential equations possess one or more linear invariants, and
these are automatically preserved by Runge-Kutta and multistep methods.
In order to also preserve a nonlinear invariant, one can use
orthogonal projection (see, e.g., \cite[Section~IV.4]{hairer2006geometric}) or relaxation (see,
e.g., \cite{ranocha2020general}).  However, orthogonal projection is problematic
in that setting because it destroys the linear invariant.  Relaxation is
advantageous in this case and can be viewed as an oblique projection, along a
line that is known to conserve linear invariants.

In this section we propose a quadratic-preserving relaxation method, analogous
to the linear-preserving relaxation method just mentioned.
It performs a projection or line search along a curve
known to preserve one simple (in this case quadratic) invariant,
in order to find a solution that also preserves a second nonlinear invariant.
Due to the simplicity of projecting onto a quadratic invariant manifold,
this leads to an efficient and robust algorithm and avoids the need to
solve more than one nonlinear algebraic equation, as was required
with the multiple relaxation approach in \cite{biswas2023multiple,biswas2024accurate}.

\subsection{Quadratic-preserving relaxation}

Consider an ODE
\begin{equation}
\label{eq:ode}
  \frac{\dif}{\dif t} u(t) = f\bigl( u(t) \bigr)
\end{equation}
with invariant (first integral) $\eta$, i.e.,
\begin{equation}
  \forall u \colon \quad \eta'(u) f(u) = 0.
\end{equation}
Relaxation one-step methods
\cite{ketcheson2019relaxation,ranocha2020relaxation,ranocha2020general} perform
the following steps to conserve $\eta$:
\begin{enumerate}
  \item Given $u^n \approx u(t^n)$, compute a provisional value
        $\tilde{u}^{n+1} \approx u\bigl( \tilde{t}^{n+1} \bigr)$ using a standard
        one-step method, where $\tilde{t}^{n+1} = t^n + \Delta t$.
  \item Solve the scalar equation
        \begin{equation}
          \eta\bigl( u^n + \gamma (\tilde{u}^{n+1} - u^n) \bigr)
          =
          \eta(u^n)
        \end{equation}
        for the scalar relaxation parameter $\gamma$.
  \item Continue the integration with
        \begin{equation}
          u^{n+1} = u^n + \gamma \bigl( \tilde{u}^{n+1} - u^n \bigr)
          \approx u\bigl( t^{n+1} \bigr),
          \qquad
          t^{n+1} = t^n + \gamma \Delta t,
        \end{equation}
        instead of $\tilde{u}^{n+1}$ and $\tilde{t}^{n+1}$.
\end{enumerate}
The available theory guarantees that there is a unique solution
$\gamma = 1 + \mathcal{O}(\Delta t^{p-1})$ under a certain non-degeneracy
condition, where $p \ge 2$ is the order of the baseline method; in
this case, the relaxed method is also at least $p$-th order accurate
\cite{ketcheson2019relaxation,ranocha2020relaxation,ranocha2020general}.

By construction, the relaxed method conserves the invariant $\eta$
and also all linear invariants conserved by the baseline method, i.e.,
all linear invariants of the ODE for general linear methods such as
Runge-Kutta methods.
Since the NLS has no linear invariants
but two nonlinear invariants (mass and energy) that are conserved by
the spatial discretizations introduced in
Section~\ref{sec:spatial_semidiscretizations},
we generalize the relaxation approach to this setting.
Instead of searching for a solution conserving the invariant $\eta$
along the secant line connecting $u^n$ and $\tilde{u}^{n+1}$,
we propose to perform the search along the geodesic line connecting
$u^n$ and $\tilde{u}^{n+1}$ on the manifold defined by the total mass,
i.e., a sphere. By construction, this approach conserves the total mass
and the total energy.

To describe the method, we consider a slightly more general setting.
Assume that we have an ODE \eqref{eq:ode} with two invariants
$\eta$ and $\mu$. Moreover, assume that the projection onto the
manifold defined by the invariant $\mu$ is given by
the (general, non-linear) projection operator $\pi$.
The new quadratic-preserving relaxation method performs the following steps:
\begin{enumerate}
  \item Given $u^n \approx u(t^n)$, compute a provisional value
        $\tilde{u}^{n+1} \approx u\bigl( \tilde{t}^{n+1} \bigr)$
        using a standard one-step method, where
        $\tilde{t}^{n+1} = t^n + \Delta t$.
  \item Project the provisional value onto the manifold defined by
        the projection operator $\pi$:
        \begin{equation}
          \hat{u}^{n+1} = \pi\bigl( \tilde{u}^{n+1} \bigr).
        \end{equation}
  \item Solve the scalar equation
        \begin{equation}
          \eta\Bigl( \pi \bigl( u^n + \gamma \bigl( \hat{u}^{n+1} - u^n \bigr) \bigr) \Bigr)
          =
          \eta(u^n)
        \end{equation}
        for the scalar relaxation parameter $\gamma$.
  \item Continue the integration with
        \begin{equation}
          u^{n+1} = \pi \bigl( u^n + \gamma \bigl( \hat{u}^{n+1} - u^n \bigr) \bigr)
          \approx u\bigl( t^{n+1} \bigr),
          \qquad
          t^{n+1} = t^n + \gamma \Delta t,
        \end{equation}
        instead of $\tilde{u}^{n+1}$ and $\tilde{t}^{n+1}$.
\end{enumerate}
For the nonlinear Schr{\"o}dinger equation, the projection operator $\pi$
is the projection onto the sphere defined by the total mass of $u^{n}$, i.e.,
\begin{equation}
  \pi\bigl( \tilde{u}^{n+1} \bigr)
  =
  \sqrt{\frac{\mass\bigl( u^{n} \bigr)}
             {\mass\bigl( \tilde{u}^{n+1} \bigr)}} \tilde{u}^{n+1}.
\end{equation}
This approach is justified by the following theorem.
\begin{theorem}
\label{thm:geodesic_relaxation}
  Assume that the ODE \eqref{eq:ode} has the invariants $(\eta, \mu)$
  and that $\pi$ is the projection operator onto the manifold defined
  by the invariant $\mu$. Moreover, assume that the baseline one-step
  method is of order $p \ge 2$ and that
  \begin{equation}
  \label{eq:nondegeneracy_condition}
    \eta'\bigl( u^n \bigr) \pi'\bigl( u^n \bigr) f'\bigl( u^n \bigr) f\bigl( u^n \bigr)
    \ne 0.
  \end{equation}
  Then, the quadratic-preserving relaxation method described above is well-defined
  for sufficiently small time step sizes $\Delta t$; there is a
  unique solution $\gamma = 1 + \mathcal{O}(\Delta t^{p-1})$ and
  the quadratic-preserving relaxation method is at least of order $p$
  (when measuring the error at the relaxed time $t^{n+1}$).
  Moreover, it conserves both invariants $\eta$ and $\mu$.
\end{theorem}

\begin{remark}
  The non-degeneracy condition \eqref{eq:nondegeneracy_condition}
  is the same as for standard relaxation methods if
  $\pi = \operatorname{id}$ is the identity. It guarantees, for example,
  that $u^n$ is not a steady state of the ODE (in which case the baseline
  method would not change the solution and $\gamma$ would be undefined).
  Since $\eta\bigl( \pi(u) \bigr) \equiv \mathrm{const}$ (along the exact flow), we have
  \begin{equation}
    \eta' \pi' f = 0,
    \quad
    \eta''(\pi' f, f) + \eta' \pi''(f, f) + \eta' \pi' f' f = 0.
  \end{equation}
  Hence, for $\pi = \operatorname{id}$, the non-degeneracy condition
  \eqref{eq:nondegeneracy_condition} is equivalent to
  $\eta' f' f = - \eta''(f, f) \ne 0$. This is satisfied, for example,
  if $\eta$ is strictly convex and $f\bigl( u^n \bigr) \ne 0$.
\end{remark}

\begin{remark}
\label{rem:nondegeneracy_condition_for_nls}
  In the specific case of the nonlinear Schr{\"o}dinger equation \eqref{eq:nls-u}, steady states are not the only case when the non-degeneracy condition \eqref{eq:nondegeneracy_condition} is not satisfied.
  For example, consider a spatially homogeneous solution $u(t, x) = \alpha \exp(\i \beta |\alpha|^2 t)$ with $\alpha \in \mathbb{C}$.
  For $\alpha \ne 0$, this is not a steady state.
  In terms of the real and imaginary parts $v$ and $w$, the NLS \eqref{eq:nls-vw} reduces to a system of ODEs with right-hand side
  \begin{equation*}
    f(v, w) = \begin{pmatrix}
      -\beta (v^2 + w^2) w \\
      \beta (v^2 + w^2) v
    \end{pmatrix}.
  \end{equation*}
  Thus,
  \begin{equation*}
    f'(v, w) f(v, w)
    =
    - \beta^2 (v^2 + w^2)^2 \begin{pmatrix}
      v \\
      w
    \end{pmatrix}.
  \end{equation*}
  Moreover, the projection $\pi$ is given by
  \begin{equation*}
    \pi(v, w)
    =
    \frac{c}{\sqrt{v^2 + w^2}} \begin{pmatrix}
      v \\
      w
    \end{pmatrix},
  \end{equation*}
  where the constant $c \in \R$ is determined by the desired total mass and the volume of the spatial domain.
  Thus, we have
  \begin{equation*}
    \pi'(v, w)
    =
    \frac{c}{\sqrt{v^2 + w^2}^3} \begin{pmatrix}
      w^2 & -v w \\
      -v w & v^2
    \end{pmatrix}.
  \end{equation*}
  Therefore, $\pi' f' f = 0$, i.e., the non-degeneracy condition \eqref{eq:nondegeneracy_condition} is not satisfied for spatially homogeneous solutions.
  This can also be interpreted using the hydrodynamic formulation of the nonlinear Schr{\"o}dinger equation\footnote{Some constants are different compared to the hydrodynamic formulation as written, e.g., in \cite{dhaouadi2019extended}, since we use the normalization $u_{xx}$ instead of $\frac{1}{2} u_{xx}$ in the NLS \eqref{eq:nls-u}.}
  \begin{equation}
    \begin{aligned}
      \varrho_t + 2 (\varrho V)_x &= 0, \\
      V_t + 2 V V_x - \beta \varrho_x - \frac{1}{2} \left( \frac{\varrho_{xx}}{\varrho} - \frac{(\varrho_x)^2}{2 \varrho^2} \right)_x &= 0,
    \end{aligned}
  \end{equation}
  based on the Madelung transformation $u = \sqrt{\varrho} \exp(\i \theta)$, $V = \theta_x$.
  The pure phase-rotation solutions correspond to a constant density $\varrho$ and zero velocity $V$, i.e., a steady state of the hydrodynamic formulation, which explains the failure of the non-degeneracy condition.
  For a genuinely non-steady solution, we expect the non-degeneracy condition to be satisfied in general.
\end{remark}

\begin{proof}[Proof of Theorem~\ref{thm:geodesic_relaxation}]
  The proof is based on the implicit function theorem, similar to
  proofs for standard projection and relaxation methods, e.g.,
  \cite[Theorem~2]{calvo2010projection},
  \cite[Proposition~2.18]{ranocha2020relaxation}, and
  \cite[Theorem~2.14]{ranocha2020general}.
  Consider the residual
  \begin{equation}
    \eta\Bigl( \pi \bigl( u^n + \gamma \bigl( \hat{u}^{n+1} - u^n \bigr) \bigr) \Bigr)
    - \eta(u^n),
  \end{equation}
  which needs to be zero for the desired relaxation parameter $\gamma$.
  Writing $\gamma = 1 + \delta \Delta t^{p-2}$, we can formulate this
  as the equivalent problem of finding a root $\delta$ of
  \begin{equation}
    r(\Delta t, \delta)
    =
    \Delta t^{-p} \left(
      \eta\Bigl( \pi \bigl( \hat{u}^{n+1} + \delta \Delta t^{p-2} \bigl( \hat{u}^{n+1} - u^n \bigr) \bigr) \Bigr)
      - \eta(u^n)
    \right)
  \end{equation}
  depending on $\Delta t$. Since the baseline method is $p$-th order
  accurate, we have
  \begin{equation}
    \eta\Bigl( \pi \bigl( \hat{u}^{n+1} + \delta \Delta t^{p-2} \bigl( \hat{u}^{n+1} - u^n \bigr) \bigr) \Bigr)
    - \eta(u^n)
    =
    \mathcal{O}(\Delta t^{p+1}),
  \end{equation}
  if $\delta = \mathcal{O}(\Delta t)$, since the argument of $\pi$ is
  a $p$-th order approximation of $u\bigl( t^{n} + \gamma \Delta t \bigr)$,
  see \cite[Lemma~2.7]{ranocha2020general}.
  Thus,
  \begin{equation}
    r(\Delta t = 0, \delta = 0) = 0.
  \end{equation}
  Moreover, we compute
  \begin{multline}
    \partial_\delta r(\Delta t, \delta)
    =
    \Delta t^{-2} \eta'\Bigl( \pi \bigl( \hat{u}^{n+1} + \delta \Delta t^{p-2} \bigl( \hat{u}^{n+1} - u^n \bigr) \bigr) \Bigr)
    \\
    \pi'\bigl( \hat{u}^{n+1} + \delta \Delta t^{p-2} \bigl( \hat{u}^{n+1} - u^n \bigr) \bigr)
    \bigl( \hat{u}^{n+1} - u^n \bigr).
  \end{multline}
  Since the baseline method is at least second-order accurate, we have
  \begin{equation}
    \hat{u}^{n+1} - u^n
    =
    \Delta t f(u^n) + \frac{\Delta t^2}{2} f'(u^n) f(u^n) + \mathcal{O}(\Delta t^3).
  \end{equation}
  Moreover, since $\eta\bigl( \pi(u) \bigr) \equiv \mathrm{const}$, we have $\eta' \pi' f = 0$.
  Thus,
  \begin{equation}
    \partial_\delta r(\Delta t = 0, \delta = 0)
    =
    \eta'\bigl( u^n \bigr) \pi'\bigl( u^n \bigr) f'\bigl( u^n \bigr) f\bigl( u^n \bigr)
    \ne 0
  \end{equation}
  by the non-degeneracy condition \eqref{eq:nondegeneracy_condition}.
  Thus, the implicit function theorem guarantees that there is a
  unique solution $\delta = \mathcal{O}(\Delta t)$ of
  $r(\Delta t, \delta) = 0$ for sufficiently small $\Delta t$.
  This implies the existence of a unique solution
  $\gamma = 1 + \delta \Delta t^{p-2} = 1 + \mathcal{O}(\Delta t^{p-1})$.
  The order of accuracy of the quadratic-preserving relaxation method follows
  from \cite[Lemma~2.7]{ranocha2020general}.
\end{proof}

\begin{remark}
\label{rem:quadratic-preserving_relaxation}
  We call the new method analyzed in Theorem~\ref{thm:geodesic_relaxation} a \emph{quadratic-preserving relaxation method} since it performs the search along a curve that preserves a quadratic invariant (in the case of NLS, the total mass).
  However, it is not limited to quadratic invariants enforced by the projection operator $\pi$; it can be applied to any invariant for which such a projection operator is available.
  However, the projection operator will be applied many times.
  Thus, it is crucial that $\pi$ can be evaluated efficiently, which is the case for the projection onto a sphere.
  We have also applied the method to more general settings where we conserve the total mass, momentum, and energy of the NLS, the Korteweg-de Vries (KdV) equation, and the Benjamin-Bona-Mahony (BBM) equation \cite{ranocha2025conserving}.
\end{remark}

\subsection{Numerical verification of the fully-discrete invariant conservation}

To visualize the influence of the mass- and energy-preserving relaxation method in time, we consider both the two-soliton and the three-soliton setups described in
\cite{biswas2024accurate}.
We use Fourier collocation methods with $N = 2^{10}$ nodes in the domain $[-35, 35]$ for the spatial semidiscretization and the third-order method of \cite{ascher1997implicit} for the time discretization.
Following \cite{biswas2024accurate}, we visualize the numerical solutions with and without relaxation at time $t = 4.3$.

\begin{figure}[htb]
  \centering
  \includegraphics[width=\textwidth]{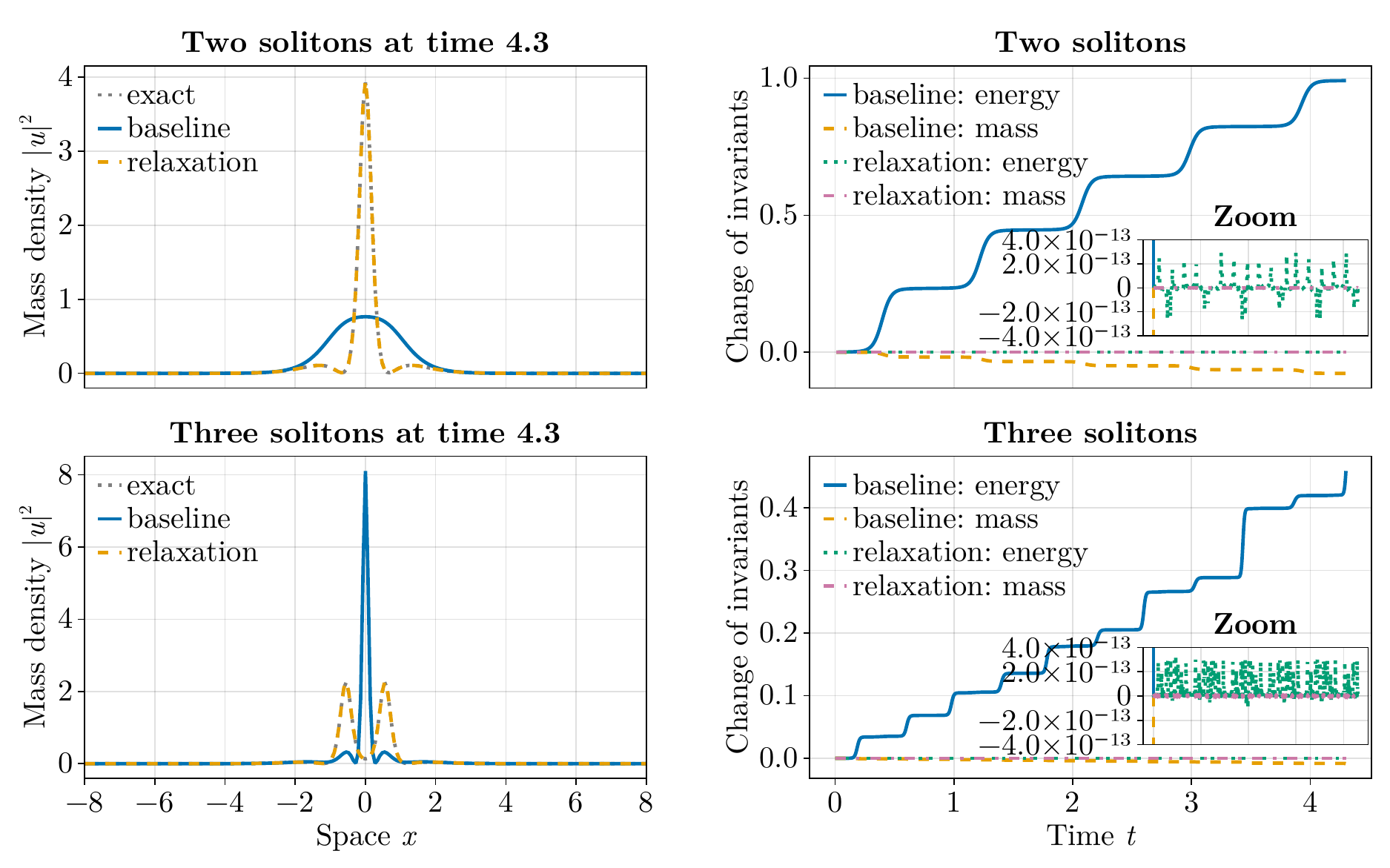}
  \caption{Numerical results obtained for the two- ($\Delta t = 10^{-2}$)
           and three-soliton ($\Delta t = 10^{-3}$)
           setups using Fourier collocation methods with $N = 2^{10}$
           nodes in the domain $[-35, 35]$ in space and the
           third-order method of \cite{ascher1997implicit} in time.}
  \label{fig:fully_discrete_conservation}
\end{figure}

The results shown in Figure~\ref{fig:fully_discrete_conservation} confirm the advantages of using relaxation to conserve mass and energy.
In particular, both mass and energy are preserved up to the error of the nonlinear solver (which is close to machine precision).
While the baseline method without relaxation does not capture the correct solution with the chosen time step sizes, the relaxed solutions are visually indistinguishable from the exact solutions.

\subsection{Convergence tests in time}

Next, we verify the high-order accuracy of the time discretizations with and without relaxation to conserve mass and energy.
Similar to \cite{biswas2024accurate}, we consider both the two-soliton and the three-soliton setups, since the one-soliton setup is easier to solve numerically.
To reduce errors of the spatial semidiscretization, we use Fourier collocation methods with $N = 2^{12}$ nodes in the domain $[-35, 35]$.

\begin{figure}[htb]
  \centering
  \includegraphics[width=\textwidth]{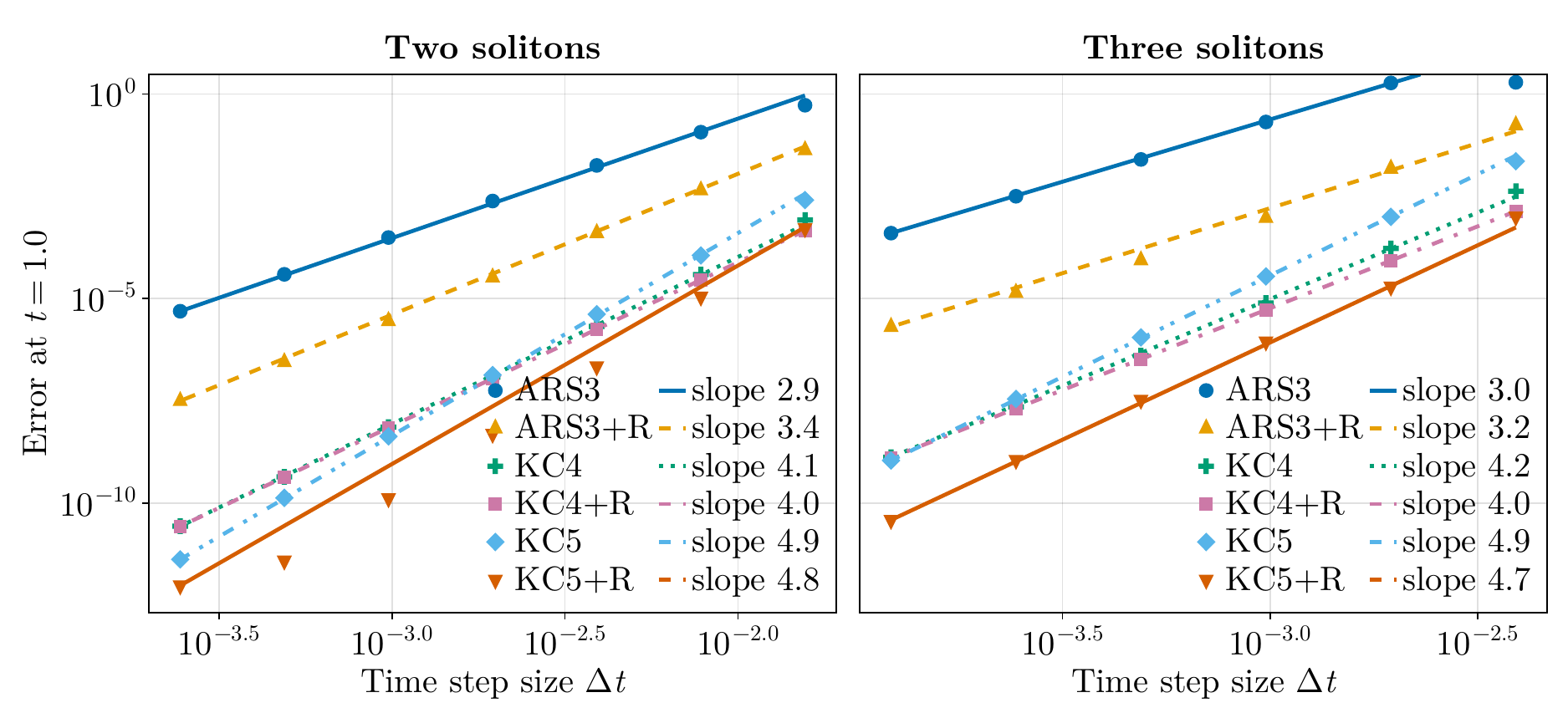}
  \caption{Temporal convergence study for the two- and three-soliton
           setups using Fourier collocation methods with $N = 2^{12}$
           nodes in the domain $[-35, 35]$ in space.
           ARS3 is the third-order method of \cite{ascher1997implicit}
           and KC4/5 are the fourth/fifth-order methods of
           \cite{kennedy2019higher}.
           ``+R'' indicates relaxation to conserve mass and energy.}
  \label{fig:convergence_in_time}
\end{figure}

The results shown in Figure~\ref{fig:convergence_in_time} confirm the high accuracy of the time discretizations.
In all cases, using relaxation to conserve mass and energy increases the accuracy.
This effect is particularly pronounced for the odd-order methods ARS3 (also known as ARS(4,4,3) of \cite{ascher1997implicit}) and KC5 (ARK5(4)8L[2]SA\textsubscript{2} of \cite{kennedy2019higher}).
Moreover, this demonstrates that the new relaxation approach works robustly.

\subsection{Error growth for multiple-soliton solutions}
\label{sec:error_growth_multiple_solitons}

In \cite{alvarez2010error} it was shown that, under certain assumptions and in the context
of the KdV equation, numerical methods that conserve $N$ invariants also exhibit linear
error growth for $N$-soliton solutions.  Here we test experimentally whether this holds
for NLS and $N=2$; we also investigate the behavior of the error for three solitons, which
goes beyond any existing theory since we have only two numerically conserved quantities.
We use the same setup as in \cite{biswas2024accurate}, employing
Fourier collocation methods with $N = 2^{10}$ nodes in the domain $[-35, 35]$ for the spatial semidiscretization and the fifth-order method of \cite{kennedy2019higher} for the time discretization.

\begin{figure}[htb]
  \centering
  \includegraphics[width=\textwidth]{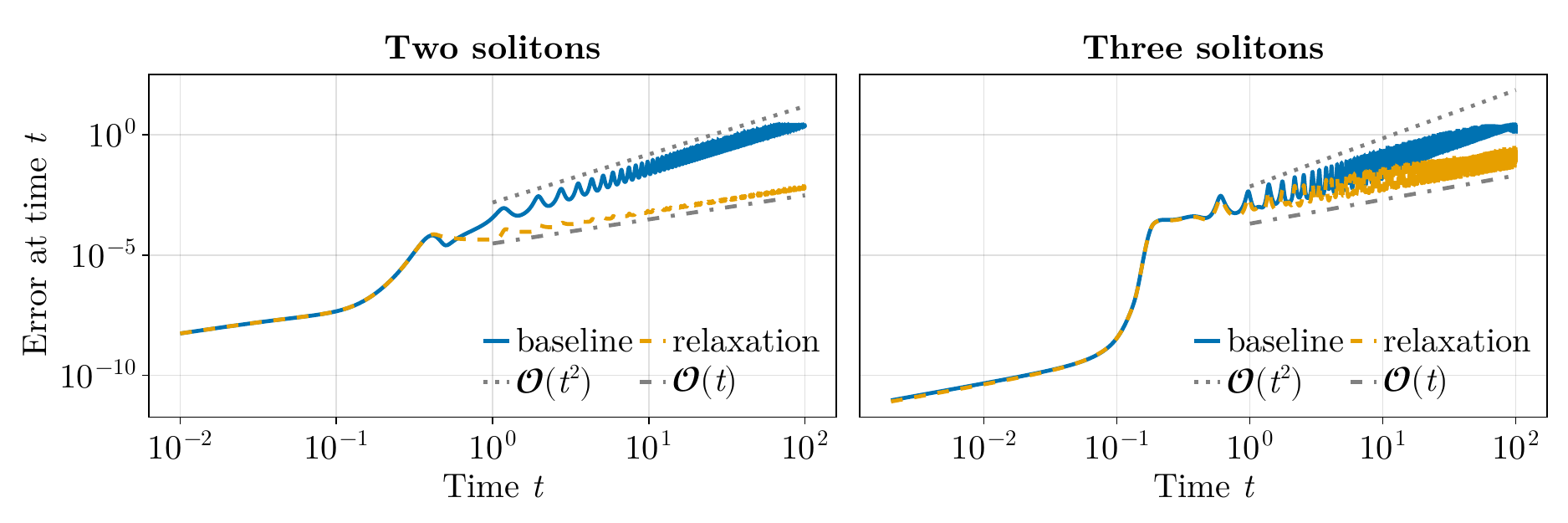}
  \caption{Error growth in time for the two- and three-soliton
           setups using Fourier collocation methods with $N = 2^{10}$
           nodes in the domain $[-35, 35]$ in space and the fifth-order
           method of \cite{kennedy2019higher} with time step size
           $\Delta t = 10^{-2}$ (two solitons) and
           $\Delta t = 2 \times 10^{-3}$ (three solitons).}
  \label{fig:error_growth}
\end{figure}

The results shown in Figure~\ref{fig:error_growth} confirm the advantages of using relaxation to conserve mass and energy.
In particular, the error grows quadratically in time for the baseline method and only linearly in time for the relaxed method.
Remarkably, we see this behavior not only for the two-soliton case (in agreement with \cite{alvarez2010error}) but
even with three solitons.

\begin{figure}[htb]
  \centering
  \includegraphics[width=\textwidth]{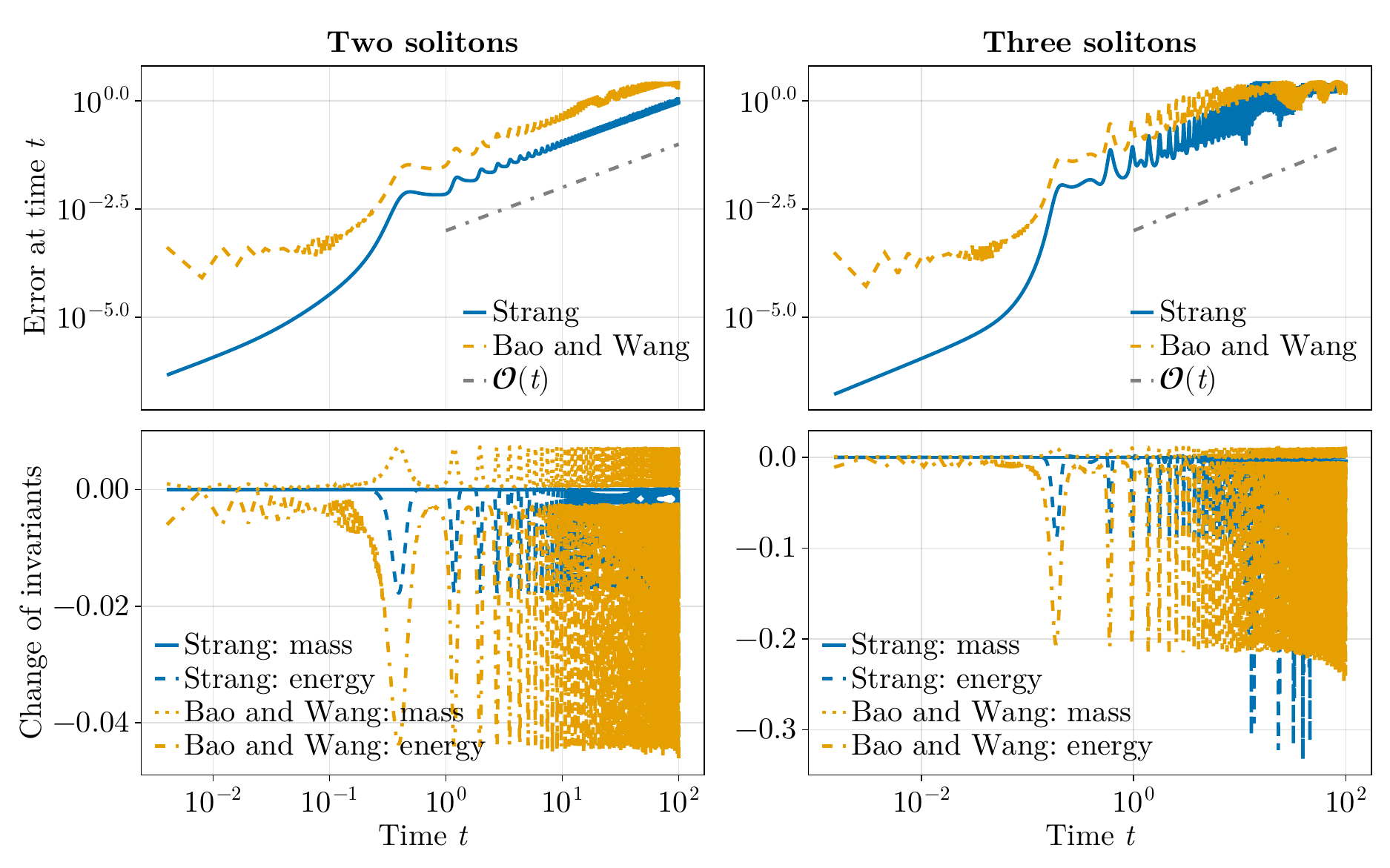}
  \caption{Error growth in time for the two- and three-soliton
           setups using Fourier collocation methods with $N = 2^{10}$
           nodes in the domain $[-35, 35]$ in space and Strang splitting
           with exact evolution operators \cite{antoine2013computational}
           as well as the exponential method of Bao and Wang \cite{bao2024explicit}
           with time step size
           $\Delta t = 4 \times 10^{-3}$ (two solitons) and
           $\Delta t = 1.5 \times 10^{-3}$ (three solitons).}
  \label{fig:error_growth_others}
\end{figure}

In addition, we show results for the widely-used Strang splitting method based on exact evolution operators for the linear and nonlinear parts of the NLS \cite{antoine2013computational} as well as the symmetric (up to the first step) exponential method of Bao and Wang \cite{bao2024explicit}, both discretized with the same Fourier collocation method in space we use for the IMEX methods.
For the Strang splitting method, we use the linear evolution operator as the inner one, since it requires an FFT and is thus more expensive than the nonlinear evolution operator, which can be evaluated pointwise.
Since one step of these two methods is cheaper than one step with the IMEX methods (which uses several stages per step), we choose smaller time step sizes so that the runtime is roughly the same as for the IMEX methods (roughly 2 seconds for the two-soliton case and 20 seconds for the three-soliton case).
The results shown in Figure~\ref{fig:error_growth_others} demonstrate that these structure-preserving methods also exhibit linear error growth in time.
However, the error of these second-order methods is (for the same computational cost) larger than for the high-order IMEX methods with quadratic-preserving relaxation.
The good error-growth behavior of the Strang splitting method and the exponential method of Bao and Wang \cite{bao2024explicit} is related to their ability to nearly preserve the mass and energy; in fact, Strang splitting with exact evolution operators even conserves the total mass \cite{antoine2013computational}.

\subsection{Error growth for a gray soliton}

Next, we consider the defocusing case $\beta = -1$.
Following \cite{dhaouadi2019extended}, we consider the gray soliton
expressed in terms of the hydrodynamic variables
\begin{equation}
  \varrho = b_1 - \frac{b_1 - b_2}{\cosh^2\Bigl(\sqrt{b_1 - b_2}\bigl(x / \sqrt{2} - c t\bigr)\Bigr)}, \quad
  V = c - \frac{b_1 \sqrt{b_2}}{\varrho},
\end{equation}
obtained from the Madelung transformation
\begin{equation}
  u = \sqrt{\varrho} \exp(\i \theta),
  \qquad
  V = \theta_x.
\end{equation}
As in \cite[Section~5.1]{dhaouadi2019extended}, we use the parameters $b_1 = 1$, $b_2 = 1.5$, and $c = 2$.
The scaling factor $\sqrt{2}$ of $x$ is chosen to compensate for the different choice of constants in the NLS equation in \cite{dhaouadi2019extended}.

Since the gray soliton is given in terms of the hydrodynamic velocity $V$ and we use the classical variable $u$, we compute $\theta$ by integrating $V$ numerically using the adaptive Gauss-Kronrod quadrature implemented in QuadGK.jl \cite{johnson2013quadgk}.
We choose a periodic domain with left boundary $-30$ and compute the right boundary $\approx 33.9412$ such that the initial condition is periodic using the root finding algorithm of \cite{klement2014using} implemented in SimpleNonlinearSolve.jl \cite{pal2024nonlinearsolve}.

\begin{figure}[htb]
  \centering
  \includegraphics[width=\textwidth]{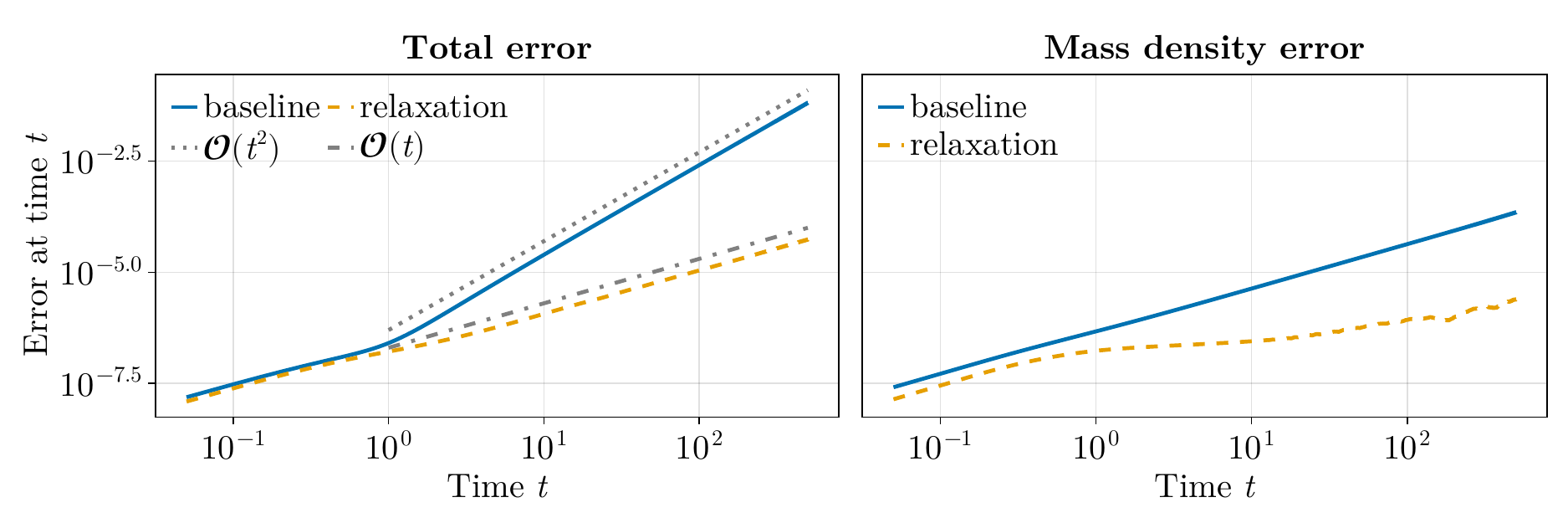}
  \caption{Error growth in time for the gray soliton
           setup using Fourier collocation methods with $N = 2^{8}$
           nodes in space and the fifth-order
           method of \cite{kennedy2019higher} with time step size
           $\Delta t = 0.05$.}
  \label{fig:error_growth_gray_soliton}
\end{figure}

The results shown in Figure~\ref{fig:error_growth_gray_soliton} confirm again the advantages of using relaxation to conserve mass and energy.
In particular, the total error grows again quadratically in time for the baseline method and only linearly in time for the relaxed method.
In contrast, the discrete $L^2$ error of the mass density $|u|^2 = v^2 + w^2$ grows only linearly in time, even without relaxation.

\section{Performance comparison}
\label{sec:performance_comparison}

Next, we compare the performance of our proposed approach to other methods
from the literature. As described in \cite[Section~5.2]{biswas2024accurate},
the best performing methods from their comparison are the linearly implicit
method of Besse et al.\ \cite{besse2004relaxation,besse2021energy} with uniform
time step sizes (limited to second-order accuracy in time) with arbitrarily
high-order finite differences in space and the multiple-relaxation method of
Biswas and Ketcheson \cite{biswas2024accurate} with time step size control
and a continuous finite element method in space (limited to second-order
accuracy in space). In their comparison, these two methods perform similarly.
Thus, we only compare our approach to the relaxation method of
Besse et al.\ \cite{besse2004relaxation,besse2021energy}.
Moreover, we include the Strang splitting \cite{antoine2013computational} and the exponential method of Bao and Wang \cite{bao2024explicit} as in Section~\ref{sec:error_growth_multiple_solitons}.

We have implemented all methods with a reasonable amount of effort to ensure a fair comparison.
For example, we have re-used the symbolic factorization of sparse matrices for the linearly implicit method of Besse et al.
Note that the system matrix changes in every time step for this method, so we have to re-assemble and re-factor the system matrix quite often.
Because of these measures (and using Julia instead of Python), our
implementations are more efficient than the ones provided by
\cite{biswas2024accurate} (roughly by a factor between three and 13 for the
results reported in \cite[Table~6]{biswas2024accurate}).

For the IMEX methods with quadratic-preserving relaxation, the system matrices do not
depend on the numerical solution but only on the time step size; we store
and re-use the sparse LU factorizations of these matrices.
We only consider Fourier methods in space for IMEX methods, where we can solve
the resulting linear systems efficiently using the fast Fourier transform.
We do not use Fourier methods for the scheme of Besse et al., since they already
noted that their method is not competitive in this case
\cite[p.~642]{besse2021energy} due to the structure of the linear systems.

\begin{figure}[htb]
  \centering
  \includegraphics[width=\textwidth]{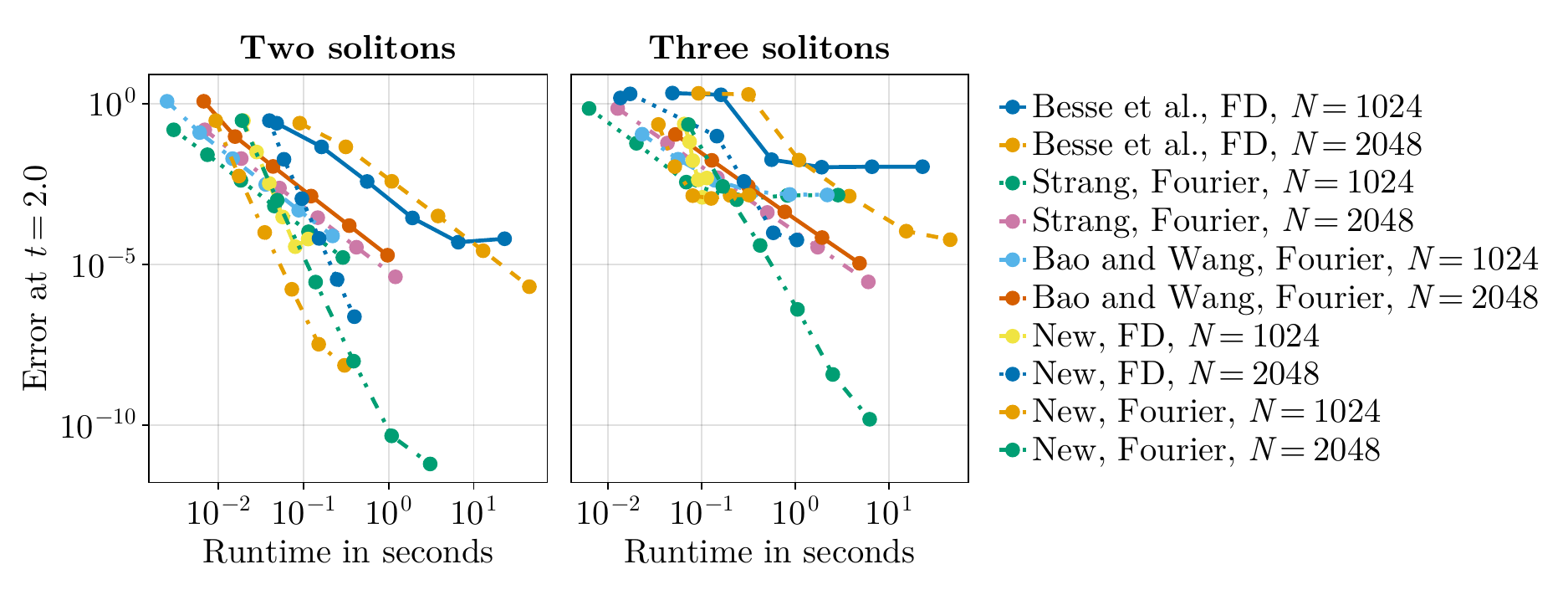}
  \caption{Error versus runtime for the two- and three-soliton setups using
           finite differences and Fourier collocation methods in the domain
           $[-35, 35]$. For our approach, we use the fifth-order method of
           \cite{kennedy2019higher} with quadratic-preserving relaxation to conserve mass
           and energy. For all methods, we used various uniform time step sizes.}
  \label{fig:performance_comparison}
\end{figure}

Following \cite{biswas2024accurate}, we consider the two- and three-soliton
setups in the spatial domain $[-35, 35]$ and compute the discrete $L^2$ error
at the final time $t = 2$. To avoid limiting the error artificially by the
spatial discretization, we use eighth-order FDs and Fourier
collocation schemes. We solve the problem three times
on an Apple M4 MacBook Pro and take the minimum runtime to reduce the influence
of noise.

The results are shown in Figure~\ref{fig:performance_comparison}.
The method of Besse et al.\ is slower than all other methods (for a given error level).
Strang splitting and the exponential method of Bao and Wang \cite{bao2024explicit} (designed for low-regularity solutions) have similar runtimes, but the Strang splitting method is more accurate for the smooth solutions considered here.
Furthermore, the usual effect of low- versus high-order methods is visible: for low accuracy requirements, Strang splitting and the method of Bao and Wang \cite{bao2024explicit} are more efficient, while for high accuracy requirements, our proposed methods are more efficient.
For our IMEX methods, Fourier collocation methods are more efficient than finite differences.

Finally, we have run a rough comparison of our methods with those of \cite{bai2024high}.
We consider the one-soliton problem from Section~4 of that work, given by
\begin{align*}
    u(x,t) & = \operatorname{sech}(x+4t) \exp(-i(2x+3t)).
\end{align*}
We solve the problem on the domain $x\in[-40,40]$ for $0\le t \le 1$ and measure
the $L^2$ norm of the error at the final time.  Using the
same discretization considered in Figure 4.4(b) of that work the wall clock time
for a run using $\Delta t = 1/512$ and 1024 points in space is approximately 90
seconds and results in an error of $1.26\times 10^{-6}$.  Meanwhile, using the
method described here with the same number of time steps and spatial points, the
wall clock time is less than 1/10 of one second and results in an error of
$9.60\times 10^{-12}$.  We observe that the method here yields a much
more accurate solution in a time that is approximately 4 orders of magnitude faster.
The tests for both methods were performed on a workstation with two sockets of
20 dual-threaded Intel Xeon Gold 6230 CPUs running Ubuntu 22.04.

\section{Hyperbolization of the nonlinear Schr{\"o}dinger equation}
\label{sec:hyperbolization}

Next, we consider the hyperbolization
\begin{equation}
\label{eq:nlsh-q}
\begin{aligned}
  \i q^0_t + q^1_x &= -\beta |q^0|^2 q^0, \\
  \i \tau q^1_t - q^0_x &= -q^1
\end{aligned}
\end{equation}
of the nonlinear Schr{\"o}dinger equation \eqref{eq:nls-u} analyzed in
\cite{biswas2025hyperbolic}. Separating the real and imaginary parts
as $q^0 = v + \i w$ and $q^1 = \nu + \i \omega$, we obtain
\begin{equation}
\label{eq:nlsh-vw}
\begin{aligned}
  v_t &= -\omega_x - \beta \bigl( v^2 + w^2 \bigr) w, \\
  w_t &= \nu_x + \beta \bigl( v^2 + w^2 \bigr) v, \\
  \tau \nu_t &= w_x - \omega, \\
  \tau \omega_t &= -v_x + \nu.
\end{aligned}
\end{equation}
This system has three nonlinear invariants approximating the mass,
momentum, and energy of the NLS
\eqref{eq:nls-u} (formally) for $\tau \to 0$ \cite{biswas2025hyperbolic}.
Mass-conserving semidiscretizations using periodic first-derivative SBP
operators have been proposed in \cite[Section~4]{biswas2025hyperbolic}\footnote{They required skew-Hermitian derivative operators in periodic domains, i.e., SBP operators with mass matrices proportional to the identity matrix such as classical central finite differences or Fourier collocation methods.}.
Here, we generalize this approach to upwind SBP operators, avoiding issues with wide-stencil second-derivative operators in the FD framework.
From the point of view of DG methods, this means we use LDG methods with alternating upwind fluxes instead of the Bassi-Rebay~1 (BR1) scheme.
Concretely, we propose the semidiscretization
\begin{equation}
\label{eq:semidiscretization-nlsh}
\begin{aligned}
  \partial_t \vec{v} &= -D_- \vec{\omega} - \beta \bigl( \vec{v}^2 + \vec{w}^2 \bigr) \vec{w}, \\
  \partial_t \vec{w} &= D_- \vec{\nu} + \beta \bigl( \vec{v}^2 + \vec{w}^2 \bigr) \vec{v}, \\
  \tau \partial_t \vec{\nu} &= D_+ \vec{w} - \vec{\omega}, \\
  \tau \partial_t \vec{\omega} &= -D_+ \vec{v} + \vec{\nu},
\end{aligned}
\end{equation}
where $D_\pm$ are periodic upwind operators. The corresponding discrete
total mass is
\begin{equation}
\label{eq:mass-nlsh}
\begin{aligned}
  \mass
  &=
  \vec{v}^T M \vec{v} + \vec{w}^T M \vec{w}
  + \tau \vec{\nu}^T M \vec{\nu} + \tau \vec{\omega}^T M \vec{\omega}
  \\
  &\approx
  \int \left( v^2 + w^2 + \tau \nu^2 + \tau \omega^2 \right) \dif x
  =
  \int \left( |q^0|^2 + \tau |q^1|^2 \right) \dif x,
\end{aligned}
\end{equation}
and the discrete total energy is
\begin{equation}
\label{eq:energy-nlsh}
\begin{aligned}
  \energy
  &=
  2 \vec{\nu}^T M D_+ \vec{v} - \vec{\nu}^T M \vec{\nu}
  + 2 \vec{\omega}^T M D_+ \vec{w} - \vec{\omega}^T M \vec{\omega}
  - \frac{\beta}{2} \vec{1}^T M \bigl( \vec{v}^2 + \vec{w}^2 \bigr)^2
  \\
  &\approx
  \int \left( 2 \nu v_x - \nu^2 + 2 \omega w_x - \omega^2 - \frac{\beta}{2} (v^2 + w^2)^2 \right) \dif x
  \\
  &=
  \int \left( \overline{q^1} q^0_x + q^1 \overline{q^0_x} - |q^1|^2 - \frac{\beta}{2} |q^0|^4 \right) \dif x.
\end{aligned}
\end{equation}
As $\tau \to 0$, we (formally) have $\vec{\nu} \to D_+ \vec{v}$
and $\vec{\omega} \to D_+ \vec{w}$ in \eqref{eq:semidiscretization-nlsh},
i.e., the semidiscretization \eqref{eq:semidiscretization-nlsh}
converges to the semidiscretization \eqref{eq:semidiscretization-nls}
of the nonlinear Schr{\"o}dinger equation \eqref{eq:nls-u}
with second-derivative SBP operator $D_2 = D_- D_+$.

\begin{theorem}
\label{thm:semidiscrete_conservation_nlsh}
  The semidiscretization \eqref{eq:semidiscretization-nlsh}
  conserves the discrete total mass \eqref{eq:mass-nlsh}
  and the discrete total energy \eqref{eq:energy-nlsh}
  for diagonal-norm upwind SBP operators.
\end{theorem}
\begin{proof}
  We compute
  \begin{equation}
  \begin{aligned}
    \frac{1}{2} \partial_t \mass
    &=
    \vec{v}^T M \partial_t \vec{v} + \vec{w}^T M \partial_t \vec{w}
    + \tau \vec{\nu}^T M \partial_t \vec{\nu} + \tau \vec{\omega}^T M \partial_t \vec{\omega}
    \\
    &=
    - \vec{v}^T M D_- \vec{\omega}
    - \beta \vec{v}^T M \bigl( \vec{v}^2 + \vec{w}^2 \bigr) \vec{w}
    + \vec{w}^T M D_- \vec{\nu}
    + \beta \vec{w}^T M \bigl( \vec{v}^2 + \vec{w}^2 \bigr) \vec{v}
    \\
    &\quad
    + \vec{\nu}^T M D_+ \vec{w} - \vec{\nu}^T M \vec{\omega}
    - \vec{\omega}^T M D_+ \vec{v} + \vec{\omega}^T M \vec{\nu}
    \\
    &=
    - \vec{v}^T M D_- \vec{\omega}
    - \vec{\omega}^T M D_+ \vec{v}
    + \vec{w}^T M D_- \vec{\nu}
    + \vec{\nu}^T M D_+ \vec{w}
    = 0,
  \end{aligned}
  \end{equation}
  where we used the periodic upwind SBP property and the fact that the mass matrix $M$ is diagonal.
  Next, we compute
  \begin{equation}
  \begin{aligned}
    &\quad
    \frac{1}{2} \partial_t \energy
    \\
    &=
    \vec{\nu}^T M D_+ \partial_t \vec{v}
    + \vec{v}^T D_+^T M \partial_t \vec{\nu}
    - \vec{\nu}^T M \partial_t \vec{\nu}
    + \vec{\omega}^T M D_+ \partial_t \vec{w}
    + \vec{w}^T D_+^T M \partial_t \vec{\omega}
    - \vec{\omega}^T M \partial_t \vec{\omega}
    \\
    &\quad
    - \beta \vec{1}^T M \bigl( \vec{v}^2 + \vec{w}^2 \bigr) \vec{v} \partial_t \vec{v}
    - \beta \vec{1}^T M \bigl( \vec{v}^2 + \vec{w}^2 \bigr) \vec{w} \partial_t \vec{w}
    \\
    &=
    - \vec{\nu}^T M D_+ D_- \vec{\omega}
    - \beta \vec{\nu}^T M D_+ \bigl( \vec{v}^2 + \vec{w}^2 \bigr) \vec{w}
    + \tau^{-1} \vec{v}^T D_+^T M D_+ \vec{w}
    - \tau^{-1} \vec{v}^T D_+^T M \vec{\omega}
    \\
    &\quad
    - \tau^{-1} \vec{\nu}^T M D_+ \vec{w}
    + \tau^{-1} \vec{\nu}^T M \vec{\omega}
    - \vec{\omega}^T M D_+ D_- \vec{v}
    + \beta \vec{\omega}^T M D_+ \bigl( \vec{v}^2 + \vec{w}^2 \bigr) \vec{v}
    \\
    &\quad
    - \tau^{-1} \vec{w}^T D_+^T M D_+ \vec{v}
    + \tau^{-1} \vec{w}^T D_+^T M \vec{\nu}
    + \tau^{-1} \vec{\omega}^T M D_+ \vec{v}
    - \tau^{-1} \vec{\omega}^T M \vec{\nu}
    \\
    &\quad
    + \beta \vec{1}^T M \bigl( \vec{v}^2 + \vec{w}^2 \bigr) \vec{v} D_- \vec{\omega}
    + \beta^2 \vec{1}^T M \bigl( \vec{v}^2 + \vec{w}^2 \bigr)^2 \vec{v} \vec{w}
    \\
    &\quad
    - \beta \vec{1}^T M \bigl( \vec{v}^2 + \vec{w}^2 \bigr) \vec{w} D_- \vec{\nu}
    - \beta^2 \vec{1}^T M \bigl( \vec{v}^2 + \vec{w}^2 \bigr)^2 \vec{v} \vec{w}
    \\
    &=
    - \vec{\nu}^T M D_+ D_- \vec{\omega}
    - \beta \vec{\nu}^T M D_+ \bigl( \vec{v}^2 + \vec{w}^2 \bigr) \vec{w}
    - \vec{\omega}^T M D_+ D_- \vec{v}
    + \beta \vec{\omega}^T M D_+ \bigl( \vec{v}^2 + \vec{w}^2 \bigr) \vec{v}
    \\
    &\quad
    + \beta \vec{1}^T M \bigl( \vec{v}^2 + \vec{w}^2 \bigr) \vec{v} D_- \vec{\omega}
    - \beta \vec{1}^T M \bigl( \vec{v}^2 + \vec{w}^2 \bigr) \vec{w} D_- \vec{\nu} \\
    &=
    0,
  \end{aligned}
  \end{equation}
  where we canceled terms in the second step and used the periodic upwind SBP property as well as the fact that the mass matrix $M$ is diagonal in the last step.
\end{proof}

The semidiscrete hyperbolization \eqref{eq:semidiscretization-nlsh}
cannot recover all semidiscretizations discussed above for the standard
nonlinear Schr{\"o}dinger equation \eqref{eq:nls-u} in the limit $\tau \to 0$.
In particular, it can only recover second-derivative operators that
can be factored as product of first-derivative upwind operators.
This excludes, for example, Fourier collocation methods with an even
number of nodes (unless the second-derivative operator uses the ``wrong''
treatment of the highest frequency mode \cite{johnson2011notes})
or narrow-stencil FD operators with order of accuracy four or higher.
However, LDG methods with alternating upwind fluxes or second-derivative
FD operators using a product of two alternating upwind FD operators
can be recovered in the limit $\tau \to 0$.

\subsection{Mass projection}

Since the mass \eqref{eq:mass-nlsh} of the hyperbolized NLS
scales the approximation variables $\nu, \omega$
by $\tau$, we have to consider a projection onto an ellipsoid instead
of a sphere. Since the exact orthogonal projection does not have a
simple closed-form solution\footnote{We have to find a root of a quartic
polynomial, for which no simple closed-form solution exists.},
we use a simplified projection algorithm as in
\cite[Section~IV.4]{hairer2006geometric} instead: we project along the
direction orthogonal to the constant-mass ellipsoid based on the current solution.
Given values $\vec{v}$, $\vec{w}$, $\vec{\nu}$, and $\vec{\omega}$,
the simplified projection scales $\vec{v}$, $\vec{w}$ by $\alpha_1$
and $\vec{\nu}$, $\vec{\omega}$ by $\alpha_2$, where
\begin{equation}
\label{eq:mass_projection_nlsh}
\begin{aligned}
  \alpha_1
  &=
  \frac{p^2 (\tau - 1) \tau^2 + \sqrt{-p^2 q^2 (\tau - 1)^2 \tau + c (q^2 + p^2 \tau^3)}}{q^2 + p^2 \tau^3},
  \\
  \alpha_2
  &=
  \frac{q^2 (1 - \tau) + \tau \sqrt{-p^2 q^2 (\tau - 1)^2 \tau + c (q^2 + p^2 \tau^3)}}{q^2 + p^2 \tau^3},
\end{aligned}
\end{equation}
$c$ is the desired value of the mass \eqref{eq:mass-nlsh}, and
\begin{equation}
  q^2 = \| \vec{v} \|_M^2 + \| \vec{w} \|_M^2,
  \qquad
  p^2 = \| \vec{\nu} \|_M^2 + \| \vec{\omega} \|_M^2.
\end{equation}

\subsection{Numerical results}

The semidiscretization \eqref{eq:semidiscretization-nlsh} of the
hyperbolized NLS behaves as expected.
The asymptotic-preserving properties can be analyzed as demonstrated
in \cite{biswas2025hyperbolic}. Here, we only present selected numerical
results demonstrating fully-discrete conservation of the total mass
\eqref{eq:mass-nlsh} and the total energy \eqref{eq:energy-nlsh}.

\begin{figure}[htb]
  \centering
  \includegraphics[width=\textwidth]{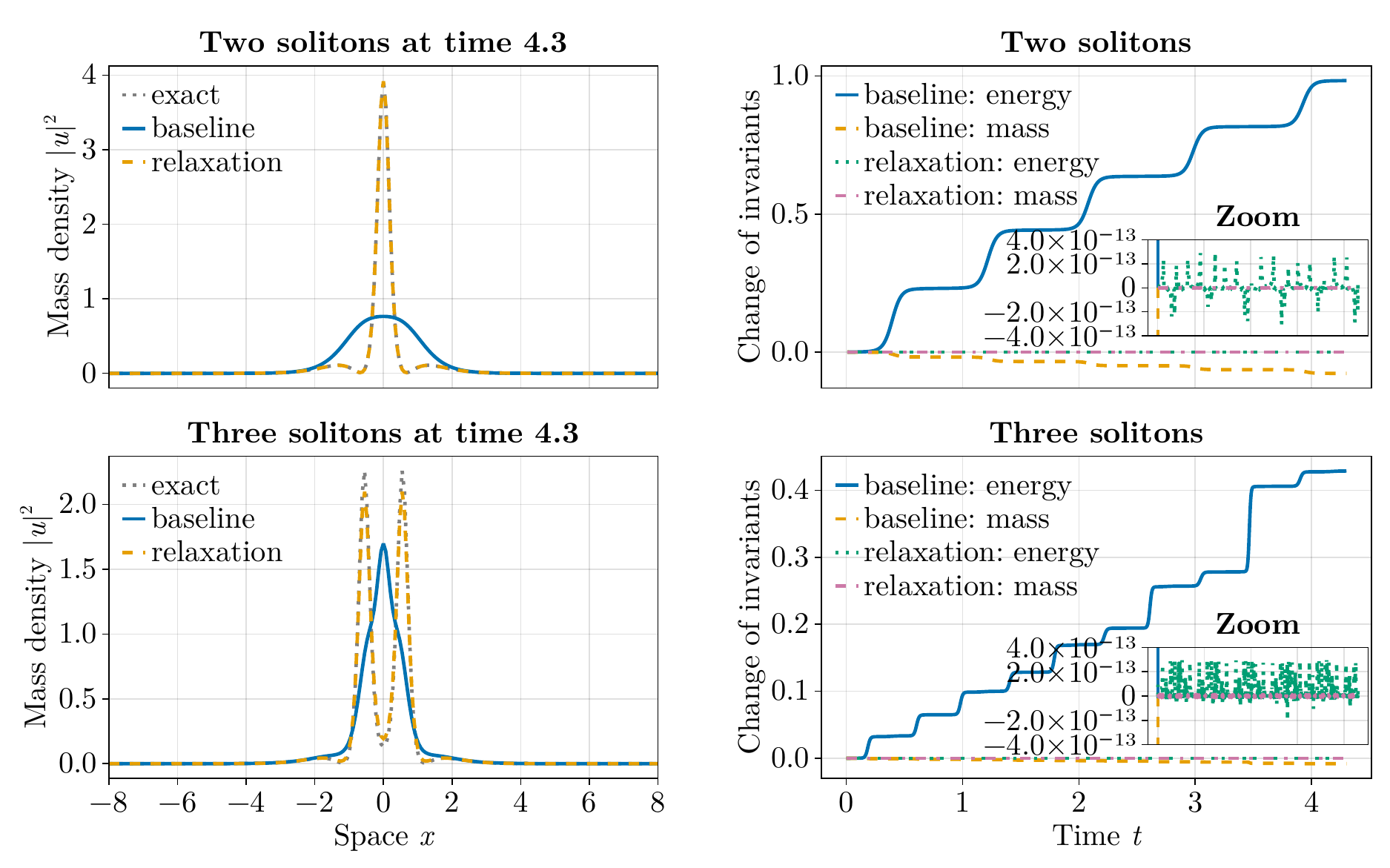}
  \caption{Numerical results obtained for the two- ($\Delta t = 10^{-2}$)
           and three-soliton ($\Delta t = 10^{-3}$)
           setups for the hyperbolized nonlinear Schr{\"o}dinger equation
           with $\tau = 10^{-4}$ using sixth-order upwind FD SBP operators
           with $N = 2^{10}$ nodes in the domain $[-35, 35]$ in space and
           the third-order method of \cite{ascher1997implicit} in time.}
  \label{fig:fully_discrete_conservation_hyperbolization}
\end{figure}

We discretize the domain $[-35, 35]$ with sixth-order upwind FD SBP
operators with $N = 2^{10}$ nodes. We use well-prepared initial data for
the hyperbolization, i.e., we use the initial values
$\vec{\nu}^0 = D_+ \vec{v}^0$ and $\vec{\omega}^0 = D_+ \vec{w}^0$.
Results with and without relaxation for are shown in
Figure~\ref{fig:fully_discrete_conservation_hyperbolization}.
As expected, the results are similar to the ones obtained with
the baseline nonlinear Schr{\"o}dinger equation, see
Figure~\ref{fig:fully_discrete_conservation}.
We clearly observe a significant improvement of the numerical solutions
by conserving the total mass and energy discretely using quadratic-preserving relaxation.

\section{Summary and discussion}
\label{sec:summary}

In this work we have shed new light on the conservation properties of existing
schemes in addition to proposing new efficient conservative schemes for the NLS equation
and its hyperbolic approximation.  Our analysis makes it clear that using
Fourier spectral differentiation in space yields a semidiscrete scheme that conserves
discrete analogs of both mass and energy, as do all methods that use a broad class of
SBP operators.

Our new quadratic-preserving relaxation method extends these conservation
properties to the fully-discrete setting, in a manner that we have shown to be
more robust and computationally efficient than the multiple-relaxation approach.
As our new approach is focused on the time discretization, it can immediately be brought to bear on multidimensional applications of NLS using tensor-product discretizations of the 1D operators used in this work, as demonstrated in \cite{rajvanshi2026efficient}.
Compared to widely-used time-splitting and exponential methods, our approach shares the usual advantages of IMEX methods \cite{antoine2016high}, i.e., their high accuracy and applicability for complex scenarios such as non-autonomous systems or general systems where no exact evolution operators are available, while also sharing the advantages of structure-preserving methods, i.e., improved long-time behavior.
For low-regularity solutions or low-accuracy requirements, specifically-designed structure-preserving methods such as exponential methods \cite{bao2024explicit} or time-splitting methods \cite{antoine2013computational} can be more efficient.

The quadratic-preserving relaxation technique can also be applied to other systems
that possess both quadratic and non-quadratic nonlinear conserved functionals.
Furthermore, it suggests a natural extension in order to conserve more than
two conserved quantities.  The analysis and application of these generalizations
are the subject of ongoing work \cite{ranocha2025conserving}.

\appendix

\section*{Acknowledgments}

HR was supported by the Deutsche Forschungsgemeinschaft
(DFG, German Research Foundation, project numbers 513301895 and 528753982
as well as within the DFG priority program SPP~2410 with project number 526031774).
DK was supported by funding from King Abdullah University of Science and Technology.

We thank Ángel Durán for hosting us in Valladolid during the week September 15--19 2025,
where this project was initiated.

\printbibliography

\end{document}